\documentclass[11pt,draft]{article}

\usepackage{mathtools}

\usepackage{amssymb, amsmath}
\usepackage{amsthm}
\usepackage{xcolor}
\usepackage{tikz}
\usepackage{url}
\usepackage{fullpage}

\newcommand{\N}{\mathbb{N}}
\newcommand{\Z}{\mathbb{Z}}

\newcommand{\R}{\mathbb{R}}
\newcommand{\C}{\mathbb{C}}

\renewcommand{\hom}{\mathsf{hom}}

\newcommand{\loc}{\mathrm{loc}}

\renewcommand{\epsilon}{\varepsilon}

\renewcommand{\i}{{\rm i}}

\newcommand{\nablao}{\overset{\mathsf{o}}{\operatorname{\rm grad}}_x}
\newcommand{\nablap}{\operatorname{\rm grad}_y^{\#}}
\newcommand{\curlo}{\overset{\mathsf{o}}{\curl}_x}
\newcommand{\divp}{\operatorname{\rm div}_y^{\#}}



\DeclareMathOperator{\curl}{curl}

\DeclareMathOperator{\dive}{div}
\DeclareMathOperator{\rge}{ran}
\DeclareMathOperator{\kar}{ker}
\DeclareMathOperator{\dom}{dom}
\DeclareMathOperator{\lin}{lin}
\DeclareMathOperator{\m}{m}
\DeclareMathOperator{\diag}{diag}
\renewcommand{\Re}{\operatorname{Re}}

\DeclareMathAccent{\Circ}{\mathalpha}{operators}{"17}

\newcommand{\unf}{\mathcal{T}_{\varepsilon}} 
\newcommand{\invunf}{\mathcal{T}_{-\varepsilon}}
\newcommand{\brac}[1]{\left({#1}\right) }
\newcommand{\cb}[1]{\left\lbrace {#1} \right\rbrace}
\newcommand{\expect}[1]{\left\langle {#1} \right\rangle}

\newcommand{\e}{{\rm e}}

\let\eps\varepsilon
\let\phi\varphi

\let\leq\leqslant
\let\ge\geqslant
\let\geq\geqslant

\makeatletter
\def\@row#1,{#1\@ifnextchar;{\@gobble}{&\@row}}
\def\@matrix{%
	\expandafter\@row\my@arg,;%
	\@ifnextchar({\\ \get@in@paren{\@matrix}}{\after@matrix}%
}
\def\matrixtype#1#2#3{%
	\ifmmode\def\after@matrix{\end{#2}\right#3}%
\else\def\after@matrix{\end{#2}\right#3$}$\fi\iffalse$\fi
\left#1\begin{#2}\get@in@paren{\@matrix}$%
}
\def\@column#1,{#1\@ifnextchar;{\@gobble}{\\ \@column}}
\newcommand\vect{}
\def\svect(#1){\left(\begin{smallmatrix}\@column#1,;\end{smallmatrix}\right)}
\def\vect{\get@in@paren{\@vect}}
\def\@vect{\left(\begin{matrix}\expandafter\@column\my@arg,;\end{matrix}\right)}
\def\get@in@paren#1({\def\my@arg{}\def\my@rest{}\def\after@get{#1}\get@arg}
\let\e@a\expandafter
\def\get@arg#1){\e@a\kl@test\my@rest#1(;}
\def\kl@test#1(#2;{\e@a\def\e@a\my@arg\e@a{\my@arg#1}%
	\ifx:#2:\let\my@exec\after@get
	\else\let\my@exec\get@arg
	\e@a\def\e@a\my@arg\e@a{\my@arg(}%
	\def@rest#2;%
	\fi\my@exec}
\def\def@rest#1(;{\def\my@rest{#1\kl@zu}}
\def\kl@zu{)}

\makeatletter
\newcommand\MyPairedDelimiter{%
	\@ifstar{\My@Paired@Delimiter{{}}}
	{\My@Paired@Delimiter{}}%
}
\newcommand\My@Paired@Delimiter[4]{%
	\newcommand#2{%
		\@ifstar{\start@PD{#1}{\delimitershortfall=-1sp}{#3}{#4}}
		{\start@PD{#1}{}{#3}{#4}}%
	}%
}
\newcommand\start@PD[5]{%
	#1\mathopen{\mathpalette\put@delim@helper{\put@delim{#2}{#3}{.}{#5}}}%
	#5%
	\mathclose{\mathpalette\put@delim@helper{\put@delim{#2}{.}{#4}{#5}}}%
}
\newcommand\put@delim@helper[2]{%
	\hbox{$\m@th\nulldelimiterspace=0pt #2#1$}%
}
\newcommand\put@delim[5]{%
	\setbox\z@\hbox{$\m@th#5{#4}$}%
	\setbox\tw@\null
	\ht\tw@\ht\z@ \dp\tw@\dp\z@
	#1#5%
	\left#2\box\tw@\right#3%
}

\makeatother
\MyPairedDelimiter*{\abs}{\lvert}{\rvert}
\MyPairedDelimiter*{\norm}{\lVert}{\rVert}
\MyPairedDelimiter{\set}{\{}{\}}

\theoremstyle{plain} % default
\newtheorem{theorem}{Theorem}[section]
\newtheorem{corollary}[theorem]{Corollary}
\newtheorem{lemma}[theorem]{Lemma}
\newtheorem{proposition}[theorem]{Proposition}
\theoremstyle{definition}

\newtheorem{definition}[theorem]{Definition}
\newtheorem{remark}[theorem]{Remark}
\newtheorem{assumption}{Assumption}

\usepackage{enumitem}

\setenumerate[1]{nolistsep} 
\setenumerate[2]{nolistsep} 

\begin{document}

\medmuskip=4mu plus 2mu minus 3mu
\thickmuskip=5mu plus 3mu minus 1mu
\belowdisplayshortskip=9pt plus 3pt minus 5pt

\title{Two-scale homogenization of abstract linear time-dependent PDEs}

\author{Stefan Neukamm, Mario Varga and Marcus Waurick}

\date{}

\maketitle

\begin{abstract}
 Many time-dependent linear partial differential equations of mathematical physics and continuum mechanics can be phrased in the form of an abstract evolutionary system defined on a Hilbert space. In this paper we discuss a general framework for homogenization (periodic and stochastic) of such systems. The method combines a unified Hilbert space approach to evolutionary systems with an operator-theoretic reformulation of the well-established periodic unfolding method in homogenization. Regarding the latter, we introduce a well-structured family of unitary operators on a Hilbert space that allows to describe and analyze differential operators with rapidly oscillating (possibly random) coefficients. We illustrate the approach by establishing periodic and stochastic homogenization results for elliptic partial differential equations, Maxwell's equations, and the wave equation.
\end{abstract}
\textbf{Keywords:} periodic and stochastic homogenization, unfolding, abstract evolutionary equations, Maxwell's equations

\tableofcontents
\section{Introduction}
In this work we propose an abstract strategy for homogenization of evolutionary equations. The methods we present are based on a combination of three concepts: The abstract unified theory for evolutionary equations introduced in \cite{PicPhy} (see also \cite{MPTW14_PDE_H,A11}), their homogenization theory developed in \cite{W11_P} (with plenty of applications, see references below), and the stochastic unfolding procedure recently introduced in \cite{neukamm2018stochastic,heida2019stochastic}. 

According to \cite{PicPhy} (see also \cite{A11,PTW15_WP_P,W16_H} and the references therein) a large variety of linear evolutionary equations of mathematical physics can be recast in the following form:
\begin{equation*}
\brac{\mathcal{M}(\partial_{t,\nu})+A}u=F,
\end{equation*}
where $z\mapsto \mathcal{M}(z)$ denotes a family of linear bounded operators on a Hilbert space $H$ and $A$ is a skew-selfadjoint densely defined operator on $H$; $\mathcal{M}(\partial_{t,\nu})$ is defined in the sense of an explicit functional calculus for $\partial_{t,\nu}$, the time derivative, which is established as a boundedly-invertible operator. We refer to \cite{PicPhy} (see also \cite{SW16_SD,W16_H,A11,Trostorff2015a,Trostorff2013}) for existence and other essential results regarding the above equation. We briefly recall the most important ingredients of this framework in Section \ref{sec:ahil}.   

In many applications one is interested to describe physical properties of systems  (e.g., composites, alloys, metamaterials) that feature material heterogeneities on a small length (or time) scale, say $\varepsilon\ll 1$. 
In the above described framework this may be described by equations of the form
\begin{equation}\label{eq:240}
\brac{\mathcal{M}_{\varepsilon}(\partial_{t,\nu})+ A}u= F,
\end{equation} 
where $\mathcal{M}_{\varepsilon}$ denotes a sequence of operators with ``coefficients oscillating on scale $\varepsilon$''. The goal of homogenization is to derive a simplified, \textit{effective} equation, say an equation of the form 
\begin{equation}\label{eq:244}
\brac{\mathcal{M}_{0}(\partial_{t,\nu})+A}u= F,
\end{equation}
 that captures the large scale properties of the original system. This is typically achieved by studying the asymptotics $\varepsilon\to 0$, and by proving that the solution of (\ref{eq:240}) converges in a suitable sense to the solution of the effective problem (\ref{eq:244}) as $\varepsilon\to 0$. {In the context of abstract evolutionary equations this has been studied in \cite{W11_P,W12_HO,W13_HP,W14_G,W14_FE,W16_HPDE,W16_H,Waurick2017,W18_NHC}. Apart from \cite{W12_HO,W14_G,W18_NHC}, the results obtained in these works are typically $H$-compactness statements, that is, it is shown that given a family of operators $(\mathcal M_\varepsilon)$ (in a suitable class), it is possible to find (i) a subsequence along which homogenization occurs and (ii) a ``homogenized'' operator $\mathcal M_0$ that appears in the homogenized equation and that possibly belongs to a larger class. In applications (e.g., modelling of microstructured materials) the operator $\mathcal M_\varepsilon$ typically originates from a partial differential equation with coefficients that rapidly oscillate on scale $\varepsilon$ and that encode the specific form of the material's microstructure. In the references \cite{W12_HO,W14_G,W18_NHC} for particular equations,  some $H$-convergence results have been obtained (using certain properties of the microstructure providing explicit formulas for $\mathcal M_0$) rather than $H$-compactness, only. Anyhow, it is desirable to incorporate \textit{information on the microstructure} into the \textit{general abstract} operator-theoretic approach, and to establish a microstructure-properties relation that allows to uniquely characterize the homogenized operator $\mathcal M_0$ based on the information on the microstructure.  }

In the classical field of homogenization of partial differential equations, the microstructure-properties relation often comes in form of homogenization formulas based on correctors and cell-problems. Moreover, various homogenization methods exist that explicitly exploit structural properties of the microstructure. In particular, two-scale convergence methods are convenient for problems involving periodic or random coefficients oscillating on a small (physical) scale. Periodic two-scale convergence \cite{Nguetseng1989,Allaire1992,lukkassen2002two} and the method of periodic unfolding \cite{cioranescu2002periodic,visintin2006towards,cioranescu2008periodic,Mielke2008} are well-suited for problems involving periodic coefficients. For stochastic homogenization problems (e.g., for models describing random heterogeneous materials) the notion of stochastic two-scale convergence \cite{Bourgeat1994,andrews1998stochastic,zhikov2006homogenization,heida2011extension} and the stochastic unfolding method \cite{heida2019stochastic,neukamm2018stochastic} are available as convenient tools.

In this paper we join the idea of the stochastic/periodic unfolding procedure with the homogenization approach for abstract evolutionary equations of the form (\ref{eq:240}) (that is we combine \cite{Papanicolaou1979,Bourgeat1994,neukamm2018stochastic} with the ideas rooted in \cite{W11_P,W12_HO}). In particular, we introduce an abstract family of operators which present a generalization of the stochastic unfolding operator (see Section \ref{s:gr}). Upon assuming a set of structural assumptions for the abstract unfolding operator, we derive a suitable homogenization result for a system of form (\ref{eq:240}). A merit in this procedure (in contrast to earlier works on abstract evolutionary equations) is that the structural conditions we assume allow us to incorporate information on the microstructure and to obtain an explicit description of the  effective model. The abstract homogenization result we obtain covers a large variety of problems. To illustrate this, we specifically reconsider periodic homogenization of elliptic partial differential equations, and obtain (as simple corollaries of our abstract theorem) stochastic homogenization results for the Maxwell's equations and the wave equation. Additionally, we derive corrector type results for the considered examples that are based on specific properties of the partial differential equation under consideration.

\textbf{Structure of the paper.} In order to motivate the theory, in Section \ref{sec:mot} we recall the standard setting for stochastic homogenization and provide the definition of stochastic unfolding. Section \ref{s:gr} provides the definitions of the unfolding operator and two-scale convergence from an abstract point of view. In the rest of that section we present some important properties of the latter notions. Section \ref{sec:4} is devoted to the derivation of a homogenization result for an abstract elliptic type problem. In Section \ref{sec:ahil} we briefly recall the setting for abstract evolutionary equations. Section \ref{sec:appl} provides a homogenization result for an abstract evolutionary equation. Section \ref{sec:exam} treats some particular examples of the previously discussed abstract theory. 

\subsection{Motivation: Two-scale homogenization}\label{sec:mot}
In this section we recall some classical results and concepts from periodic and stochastic homogenization of second order, divergence-form operators with uniformly elliptic coefficients. Our intention is to motivate some ideas and concepts of the operator-theoretic framework that we develop in the present paper. To fix ideas, let $Q\subset\R^n$ denote an open, bounded domain, $f\in L^2(Q)$,  and let $u_\eps\in H^1_0(Q)$ denote the unique weak solution to 
\begin{equation}\label{eq:1:1}
  -\nabla \cdot (a(\tfrac{\cdot}{\eps})\nabla u_\eps)=f\qquad\text{in }Q,
\end{equation}
where $a:\R^n\to\R^{n\times n}$ denotes a uniformly elliptic coefficient field, i.e., $a:\R^n\to \R^{n\times n}_{\lambda,\Lambda}$ is measurable  and $\R^{n\times n}_{\lambda,\Lambda}$ denotes (for some fixed constants of ellipticity $0<\lambda\leq\Lambda$) the set of matrices $a_0\in\R^{n\times n}$ satisfying  $a_0\xi\cdot\xi\geq \lambda|\xi|^2\text{ and }|a_0\xi|\leq \Lambda|\xi|$ (for all $\xi \in \R^n$).

Periodic homogenization is concerned with the case that the coefficient field $a$ in \eqref{eq:1:1} is periodic, say $a(\cdot+k)=a(\cdot)$ a.e.~in $\R^n$ for all $k\in\Z^n$. Under this condition a classical result (e.g., \cite{murat2018h,Bensoussan1978,CioDon}) states that there exists a uniformly elliptic matrix $a_{\mathsf{hom}}\in\R^{n \times n}$  such that $\brac{u_\eps}_{\varepsilon>0}$ weakly converges in $H^1_0(Q)$ (as $\varepsilon\to 0$) to the unique weak solution $u\in H^1_0(Q)$ to the homogenized equation \[-\nabla \cdot \brac{a_{\mathsf{hom}}\nabla u}=f.\] The homogenized coefficient matrix $a_{\mathsf{hom}}$ is characterized by the formula
\begin{equation*}
  a_{\mathsf{hom}}e_i=\int_{\Box}a(\nabla\phi_i+e_i)dx \qquad (i\in \cb{1,\ldots,n}),
\end{equation*}
where $\Box:= [0,1)^n$ represents the reference cell of periodicity, $e_i\in \R^n$ denotes the $i$'th unit vector, and $\phi_i\in H^1_{\loc}(\R^n)$ denotes a periodic solution to the corrector equation
\begin{equation}\label{eq:1:2}
  -\nabla \cdot \brac{a(\nabla\phi_i+e_i)}=0\qquad\text{in }\R^n.
\end{equation}
Thanks to the periodicity of $a$, \eqref{eq:1:2} can be solved by lifting the equation to the Sobolev space of periodic functions with zero mean (where Poincar\'e's inequality and uniform ellipticity of $a$ implies coercivity of the elliptic operator $-\nabla\cdot (a\nabla)$).
The \textit{corrector equation} \eqref{eq:1:2} is a key object in homogenization theory for elliptic equations, since its solution -- the \textit{corrector} $\phi_i$ -- captures the spatial oscillations of $u_\eps$ induced by the heterogeneity of the coefficient field $a$. This can be expressed in terms of the (formal) asymptotic expansions (using Einstein's summation convention)
\begin{equation*}
  u_\eps(x)\approx u(x)+\eps\phi_i(\tfrac{x}{\eps})\partial_i u(x),\qquad \nabla u_\eps(x)\approx \nabla u(x)+\nabla\phi_i(\tfrac{x}{\eps})\partial_iu(x)
\end{equation*}
or (more precisely) in form of the two-scale convergence statement
\begin{equation*}
  \nabla u_\eps\stackrel{2}{\rightarrow} \nabla u + \nabla \phi_i \partial_i u,
\end{equation*}
where $\stackrel{2}{\rightarrow}$ denotes the notion of (strong) two-scale convergence introduced in \cite{Nguetseng1989} and further investigated in \cite{Allaire1992} (see also \cite{lukkassen2002two}) and $\nabla u+ \nabla \varphi_i \partial_i u$ denotes the function $(x,y)\mapsto \nabla u(x)+ \nabla \varphi_i(y)\partial_{i}u(x)$. Also, another method closely related to two-scale convergence is the periodic unfolding method \cite{cioranescu2002periodic,visintin2006towards,cioranescu2008periodic}, which is based on an \textit{unfolding operator} that equivalently characterizes two-scale convergence (see also \cite{Mielke2008}).

In the stochastic case the coefficient field $a$ is assumed to be a random object, i.e., $(a(x))_{x\in\R^n}$ is viewed as a family of $\R^{n\times n}_{\lambda,\Lambda}$-valued random variables. Minimial requirements for stochastic homogenization (towards a deterministic limit) are \textit{stationarity} and \textit{ergodicity}. The former means that for any $x_1,\ldots,x_N\in \R^n$ and $z\in\R^n$, the distribution of $(a(x_1+z),\ldots,a(x_N+z))$ is independent of the shift $z\in\R^n$. Ergodicity means that any \textit{shift-invariant} (measurable) subset of the coefficients has probability $1$ or $0$. In their seminal work \cite{Papanicolaou1979}, Papanicolaou and Varadhan rephrased these conditions in an analytic framework that by now became a standard in stochastic homogenization. In the following we recall their framework and some objects involved in their approach. We refer to \cite{Papanicolaou1979,jikov2012homogenization} for details (see also the lecture notes \cite{neukamm2018introduction} for a self-contained presentation). The idea is to equip 
$\Omega:=\{a:\R^n\to\R^{n\times n}_{\lambda,\Lambda}\,:\,\text{$a$ measurable}\},$ the set of all uniformly elliptic coefficient fields, with a probability measure $\mu$ and to draw $a$ randomly from $\Omega$ according to $\mu$. The precise setup is as follows:

\begin{assumption}\label{assumpt:267} Let $(\Omega,\Sigma,\mu)$ denote a probability space with a countably generated $\sigma$-algebra (that implies separability of $L^2(\Omega)$), and let $\tau=\{\tau_x: \; x\in\R^n \}$ be a group of measurable bijections $\tau_x:\Omega\to \Omega$ such that:
\begin{enumerate}
\item (Group property). $\tau_0=\textrm{id}_\Omega$ and $\tau_{x+y}=\tau_x\circ \tau_y$ for all $x,y\in \R^n$.
\item (Measure preservation). $\mu(\tau_x A)=\mu(A)$ for all $A\in \Sigma$ and $x\in \R^n$.
\item (Measurability). $(\omega,x)\mapsto \tau_{x}\omega$ is $(\Sigma \otimes \mathcal{L}(\R^n),\Sigma)$-measurable ($\mathcal{L}(\R^n)$ denotes the Lebesgue-$\sigma$-algebra on $\R^n$).
\end{enumerate}
\end{assumption}
Let $a_0:\Omega\to\R^{n\times n}_{\lambda,\Lambda}$ be measurable, then $a(x,\omega):=a_0(\tau_x\omega)$ defines a uniformly elliptic, random coefficient field that is stationary (thanks to property (b)). Moreover, in this framework, the assumption of ergodicity reads: any shift-invariant set $A\in \Sigma$ (i.e., $\tau_x A\subset A$ for all $x\in\R^n$) satisfies $\mu(A)\in\{0,1\}$.

Under the conditions of Assumption \ref{assumpt:267} and ergodicity, Papanicolaou and Varadhan proved the following homogenization result: For $f\in L^2(Q)$ and $\eps>0$, let $u_\eps\in H^1_0(Q) \otimes L^2(\Omega)$ denote the unique (Lax-Milgram) solution to 
\begin{equation}\label{eq:1:1stoch}
  -\nabla \cdot (a_0(\tau_{\frac{x}{\varepsilon}}\omega)\nabla u_\eps(x,\omega))=f(x)\qquad\text{in } Q\times \Omega.
\end{equation}
Then there exists a uniformly elliptic coefficient matrix $a_{\hom}\in\R^{n\times n}$ (only depending on $\mu$ and $a_0$) such that $u_\eps$ weakly converges in $H^1_0(Q)\otimes L^2(\Omega)$ to the unique solution $u\in H^1_0(Q)$ of the homogenized equation $-\nabla \cdot (a_{\hom}\nabla u)=f$ in $Q$. The proof of Papanicolaou and Varadhan is based on Tartar's method of oscillating test functions and the main difficulty is to give sense to the corrector equation \eqref{eq:1:2} and the corrector $\varphi_i$ in the stochastic case.

An alternative proof of the above homogenization result was introduced in \cite{Bourgeat1994}, based on a stochastic counterpart of two-scale convergence---the notion of \textit{stochastic two-scale convergence in the mean} (see also \cite{andrews1998stochastic}). In particular, a bounded sequence $(u_\eps)_{\varepsilon}$ in $L^2(Q)\otimes L^2(\Omega)$ is said to two-scale converge in the mean to a function $u\in L^2(Q)\otimes L^2(\Omega)$, if for all test functions $\eta\in C^{\infty}_{c}(Q)$ and $\varphi \in L^2(\Omega)$, 
\begin{equation}\label{eq:319}
  \int_\Omega\int_Q u_{\eps}(x,\omega)\eta(x)\varphi(\tau_{\frac{x}{\varepsilon}}\omega)\,dx\,d\mu(\omega)\to \int_\Omega\int_Q u(x,\omega)\eta(x)\varphi(\omega)\,dx\,d\mu(\omega).
\end{equation}

\textbf{Stochastic homogenization via unfolding.} Recently, in \cite{neukamm2018stochastic, heida2019stochastic} Heida and the first two authors reconsidered the notion of two-scale convergence in the mean from the perspective of an \textit{unfolding operator} (see also \cite{varga2019stochastic}). This approach is motivated by (and shares many similarities with) the well-established notion of periodic unfolding \cite{cioranescu2002periodic}. In the following we briefly recall its definition (for more detail, see \cite{heida2019stochastic}).

For $\eta \in L^2(Q)$ and $\varphi \in L^2(\Omega)$, we define
\begin{equation}\label{eq:298}
\unf (\eta \otimes \varphi)(x,\omega)= \eta(x)\varphi(\tau_{-\frac{x}{\varepsilon}}\omega).
\end{equation}
Using the measure preserving property from Assumption \ref{assumpt:267} (b), it follows that 
\[
\|\unf (\eta\otimes \varphi)\|_{L^2(Q)\otimes L^2(\Omega)}=\|\eta\otimes \varphi\|_{L^2(Q)\otimes L^2(\Omega)},
\]
and by the density of the linear span of simple tensor products 
\[
L^2(Q)\overset{a}{\otimes}L^2(\Omega):=\lin\{(x,\omega)\mapsto \phi(x)\psi(\omega); \phi\in L^2(Q),\psi\in L^2(\Omega)\}\subset L^2(Q)\otimes L^2(\Omega) \; \text{dense},
\]
$\unf$ extends to a linear isometry $\unf: L^2(Q)\otimes L^2(\Omega)\to L^2(Q)\otimes L^2(\Omega)$. Moreover, the equality $\mathcal{T}_{\varepsilon}\mathcal{T}_{-\varepsilon}=1$ on simple tensors implies that $\mathcal{T}_\varepsilon$ is in fact unitary. A simple consequence of this definition is that for a bounded sequence $u_{\varepsilon}\in L^2(Q)\otimes L^2(\Omega)$ the weak convergence $\unf u_{\varepsilon}\rightharpoonup u$ is equivalent to stochastic two-scale convergence in the mean of $u_{\varepsilon}$ to $u$ (in sense of (\ref{eq:319})). Also, note that using the stochastic unfolding operator, we might rephrase equation (\ref{eq:1:1stoch}) as 
\begin{equation}
- \nabla \cdot \brac{\invunf a_0 \unf \nabla u}= f.
\end{equation}

In the present paper we go one step further and reconsider the idea of unfolding on an abstract operator-theoretic level where  
\begin{itemize}
\item $L^2(Q)$ and $L^2(\Omega)$ are replaced by general Hilbert spaces $H_d$ and $H_s$,
\item the stochastic unfolding operator is replaced by  a family of well-structured unitary operators,
\item $-\nabla \cdot$ and $\nabla$ are replaced by densely defined closed (unbounded) linear operators $C_d^*$ and $C_d$.
\end{itemize}
Our motivation is to develop a unified, operator-theoretic approach to homogenization (in the mean) of linear evolutionary problems with periodic, quasiperiodic or random (stationary) coefficients. The abstract unfolding strategy that we propose applies to a variety of linear PDEs and we present some examples in Section \ref{sec:exam}. Here we briefly explain one of the examples---stochastic homogenization of Maxwell's equations, which we discuss in detail in Section \ref{sec:maxwell}. In particular, for $Q\subseteq \R^3$ open, we consider the following system of equations: Find $(u_{\varepsilon}, q_{\varepsilon}): \R \times Q \to \C^3 \times \C^3 $ such that
\begin{align}\label{eq:348}
\begin{split}
\partial_{t} (\eta_{\varepsilon} u_{\varepsilon}) + \sigma_{\varepsilon}u_{\varepsilon}- \curl q_{\varepsilon} & =f,\\
\partial_{t} (\mu_{\varepsilon}q_{\varepsilon})+\curl u_{\varepsilon} & = g,
\end{split}
\end{align}
where $(f,g): \R \to L^2(Q)^3\oplus L^2(Q)^3$ is a datum and the random, oscillating coefficients are given in the form $\eta_{\varepsilon}(x,\omega)=\eta_0(x,\tau_{\frac{x}{\varepsilon}}\omega)$, $\sigma_{\varepsilon}(x,\omega) = \sigma_0(x,\tau_{\frac{x}{\varepsilon}}\omega)$, $\mu_{\varepsilon}(x,\omega)= \mu_0(x,\tau_{\frac{x}{\varepsilon}}\omega)$ with $\eta_0, \sigma_0, \mu_0 \in L^{\infty}(Q\times \Omega)^{3\times 3}$ that satisfy suitable assumptions. Note that the solution depends on $\omega\in \Omega$, which we see as a random configuration of the medium, and therefore we view the solution also as a random field, i.e., we seek functions such that at (almost) each time instance $t\in \R$ satisfy $\brac{u_{\varepsilon}(t),q_{\varepsilon}(t)}\in \brac{L^2(Q)\otimes L^2(\Omega)}^3 \oplus \brac{L^2(Q)\otimes L^2(\Omega)}^3$. In fact, we phrase \eqref{eq:348} in the form of the operator equation \eqref{eq:240} given on the functional space $L^2_{\nu}(\R; \brac{L^2(Q)\otimes L^2(\Omega)}^3 \oplus \brac{L^2(Q)\otimes L^2(\Omega)}^3)$, which is an exponentially weighted $L^2$-space with a parameter $\nu \in \R$ (see Section \ref{sec:ahil}). In the limit $\varepsilon\to 0$, we derive a two-scale homogenized system, that in the case of ergodic coefficients reads: Find $(u_0,\chi_1, q_0, \chi_2) \in L^2_{\nu}(\R; L^2(Q)^3 \oplus (L^2(Q)\otimes L^2_{\mathsf{pot}}(\Omega))\oplus L^2(Q)^3 \oplus (L^2(Q)\otimes L^2_{\mathsf{pot}}(\Omega)))$ such that
\begin{align*}
\partial_{t}  (\mathbb{E}[\eta_0 (u_0+ \chi_1)]) + \mathbb{E}[\sigma_0 \brac{u_0+\chi_1}]-\curl q_0 & = f, \\
\partial_{t} ( \mathbb{E}[\mu_0 (q_0+\chi_2)])+ \curl u_0 &= g, \\
-\mathrm{div}_{\omega}\brac{ \partial_{t} (\eta_0 (u_0+ \chi_1))+ \sigma_0 (u_0+ \chi_1)} & =0, \\
-\mathrm{div}_{\omega} \brac{\partial_{t} (\mu_0 (q_0+ \chi_2))} & = 0,  
\end{align*}
where $\mathbb{E}[\cdot]$ denotes the mathematical expectation in $(\Omega,\Sigma, \mu)$ and $\mathrm{div}_{\omega}$ is the stochastic divergence, which is defined in Section \ref{sec:ex2}. We may view the first two equations as the effective Maxwell system for the deterministic variables $\brac{u_0,q_0}$ and the correctors $\brac{\chi_1,\chi_2}$, that account for the microstructure evolution in the material, are determined by the last two corrector equations, cf. Remark \ref{rem:1214}.  
In particular, we obtain that, as $\varepsilon\to 0$, 
\begin{equation*}
\unf(u_{\varepsilon}, q_{\varepsilon})\rightharpoonup \brac{u_0+\chi_1, q_0+\chi_2} \quad \text{weakly in } L^2_{\nu}(\R;(L^2(Q)\otimes L^2(\Omega))^{3} \oplus (L^2(Q)\otimes L^2(\Omega))^{3}).
\end{equation*}
A similar result in a periodic-stochastic setting has been obtained in \cite{tachago2017stochastic}. With our operator-theoretic approach we are able to dispose of the continuity condition on the coefficients in the slow variable. Next, we can treat highly oscillatory mixed type equations by only requiring the sum of $\sigma_0$ and $\eta_0$ to be positive. Moreover, for suitably regular right-hand sides $(f,g)$, we obtain the following corrector result:
\begin{equation*}
\| u_{\varepsilon} - \mathcal{T}_{-\varepsilon}\chi_1 - u_0 \|^2_{L^2_{\nu}(\R; L^2(Q)\otimes L^2(\Omega)^3)} + \| q_{\varepsilon}-\mathcal{T}_{-\varepsilon}\chi_2 - q_0 \|^2_{L^2_{\nu}(\R; L^2(Q)\otimes L^2(\Omega)^3)} \to 0 \quad \text{as }\varepsilon\to 0.
\end{equation*}
See Section \ref{sec:maxwell} for details.

\section{The operator-theoretic setting for unfolding}\label{s:gr}

In this section we introduce the setting for abstract two-scale convergence and provide some compactness results which will be useful in the following sections. Throughout this section, we let $H_d$ and $H_s$ be Hilbert spaces, $m,n\in\mathbb{N}$; `$d$' and `$s$' are a reminder of `deterministic' and `stochastic'. 

\paragraph{Abstract ``differential'' operators.}
We consider densely defined closed linear operators
\[ 
   {C}_d\colon \dom(C_d)\subseteq H_d^m\to H_d^n ,\qquad {C}_s\colon \dom(C_s)\subseteq H_s^m\to H_s^n.
\]
Given a Hilbert space $H$ the canonical extension of $C_d$ to operators from the Hilbert space tensor product $H_d^m\otimes H$ to $H_d^n\otimes H$ will be---as a rule---again denoted by $C_d$; and similarly for $C_s$. In the applications, we have in mind, we extend $C_d$ to the space $H_d^m\otimes H_s$ attaining values in $H_d^n\otimes H_s$ and $C_s$ as closed, densely defined linear operator from $H_d\otimes H_s^m$ to $H_d\otimes H_s^n$. Note that we will further identify $H_d^m\otimes H_s=(H_d\otimes H_s)^m=H_d\otimes H_s^m$. Thus, the extended operators $C_d$ and $C_s$ are both operators defined on (subsets of) $(H_d\otimes H_s)^m$ with values in $(H_d\otimes H_s)^n$.
With this in mind, we require the following compatibility conditions for $C_s$ and $C_d$:
\begin{alignat}{2}
\begin{aligned}\label{eq:core}
 \exists D_1\subseteq H_d \text{ dense }\forall x\in D_1, y\in H_s^m\colon x\otimes y \in \dom(C_d), \\ \exists D_2\subseteq H_d \text{ dense }\forall x\in D_2, y\in H_s^n\colon x\otimes y \in \dom(C_d^*), \\
 \exists E_1\subseteq H_s \text{ dense }\forall y\in E_1, x\in H_d^m\colon x\otimes y \in \dom(C_s), \\ \exists E_2\subseteq H_s \text{ dense }\forall y\in E_2, x\in H_d^n\colon x\otimes y \in \dom(C_s^*).
\end{aligned}
\end{alignat}
The first and third conditions are trivially satisfied, if $m=1$, the second and fourth if $n=1$.
In applications discussed later on, $D_1=D_2=C_c^\infty(Q)$ for some open $Q\subseteq \mathbb{R}^n$ (a similar choice can be made for $E_1$, $E_2$).

\begin{remark}
Note the following consequence of our notation convention. Given a bounded linear operator $T\in L(H,K)$ for $H,K$ Hilbert spaces, we have $TC_d\subseteq C_dT,\text{ and }TC_s\subseteq C_sT$, see also Lemma \ref{lem:TPPT} below.
\end{remark}

\paragraph{Unfolding family.} We call a \textit{strongly continuous} map $\mathcal{T}\colon \mathbb{R}\setminus\{0\} \to L(H_d\otimes H_s)$ taking values in the set of \textit{unitary operators} an \textit{unfolding family}, if
the following structural hypotheses are satisfied  (where the action of $\mathcal{T}_\epsilon$ on powers of $H_d\otimes H_s$  is understood ``component-wise'', re-using the notation):
\begin{align}
  \mathcal{T}_\epsilon=\mathcal{T}_{-\epsilon}^{-1}&\quad(\varepsilon\in\mathbb R\setminus\{0\}),\label{eq:inverse}\\ 
   \epsilon C_d \mathcal{T}_{-\varepsilon}\phi = \epsilon \mathcal{T}_{-\varepsilon} C_d\phi + C_s \mathcal{T}_{-\varepsilon}\phi &\quad (\phi\in \dom(C_d)\cap \dom(C_s)\subseteq (H_d\otimes H_s)^m,\varepsilon \in\mathbb{R}\setminus\{0\}), \label{eq:com1} \\
\label{eq:com2}
   \mathcal{T}_{\varepsilon}C_s\subseteq C_s\mathcal{T}_{\varepsilon}&\quad(\varepsilon \in\mathbb{R}\setminus\{0\}),\\
\label{eq:inv}
  \mathcal{T}_{\varepsilon}v=v&\quad(v\in \kar(C_s),\varepsilon \in\mathbb{R}\setminus\{0\}). 
\end{align}
We remark here that \eqref{eq:com1} (in conjunction with \eqref{eq:core}) implies,
\begin{equation}\label{eq:com1b}
 \epsilon C_d^*\mathcal{T}_{-\epsilon} \phi = \epsilon \mathcal{T}_{-\epsilon}C_d^*\phi+\mathcal{T}_{-\epsilon}C_s^*\phi\quad(\phi\in\dom(C_s^*)\cap\dom(C_d^*)\subseteq (H_d\otimes H_s)^n, \epsilon\in\mathbb{R}\setminus\{0\}).
\end{equation}
The latter will be particularly important, when we discuss time-dependent problems.
\begin{remark}\label{example:336}
  The stochastic unfolding operator introduced in Section \ref{sec:mot} together with $C_d=\nabla$, and $C_s$ denoting the stochastic gradient, satisfies the above assumptions (see Section \ref{sec:exam}).
\end{remark}
Within the above setting we define (stochastic) 2-scale convergence as follows.

\begin{definition}
  Let $(u_\varepsilon)_{\varepsilon>0}$ in $H_d\otimes H_s$. Then $(u_\varepsilon)_\varepsilon$ is said to \emph{strongly (weakly) 2-scale converge to $u\in H_d\otimes H_s$} (we also use the notation $u_{\varepsilon}\overset{2}{\to}u$ ($u_{\varepsilon}\overset{2}{\rightharpoonup}u$)), if
  \[
     \mathcal{T}_\varepsilon u_\varepsilon \to u \quad(\mathcal{T}_\varepsilon u_\varepsilon \rightharpoonup u)
  \]strongly (weakly) as $\epsilon\to 0$. 
\end{definition}

\begin{remark}
  Strictly speaking 2-scale convergence is only defined for families
  $(u_\varepsilon)_{\varepsilon>0}$. In the following, we will however
  also say that ``a subsequence of $(u_\varepsilon)_{\varepsilon>0}$
  (weakly/strongly) 2-scale converges to some $u$''. By this, we mean
  that there exists a sequence $(\epsilon_k)_k$ in $(0,\infty)$
  converging to $0$ such that
  $\mathcal{T}_{\epsilon_k}u_{\epsilon_k}\to u$ weakly as $k\to
  \infty$. As the particular subsequence will not be important in the
  considerations we are aiming for, we shall however re-use $\epsilon$ and dispense with $k$.
\end{remark}

In the following we establish various properties that we shall exploit in our abstract homogenization scheme, and highlight analogies to stochastic two-scale convergence in the mean and periodic unfolding, respectively. We begin with an abstract counter part of \cite[Theorem 3.7]{Bourgeat1994}. In order to avoid cluttered notation as much as possible, we often write $H_d\otimes H_s$ regardless of the number of components of the objects under consideration. We begin with an auxiliary result.

Unless explicitly stated otherwise, we shall \emph{not} assume condition \eqref{eq:inv} in this section. The conditions \eqref{eq:core}--\eqref{eq:com2}, however, are assumed to be in effect.

\begin{lemma}\label{lem:TPPT} Let $P$ be the orthogonal projection onto $\kar(C_s)\subseteq (H_d\otimes H_s)^m$. Then $P\mathcal{T}_\epsilon=\mathcal{T}_\epsilon P$ for all $\epsilon\in\mathbb{R}\setminus\{0\}$.
\end{lemma}
\begin{proof}
  Let $\epsilon\in\mathbb{R}\setminus\{0\}$. Then we have $\mathcal{T}_\epsilon C_s\subseteq C_s \mathcal{T}_\epsilon$ by (\ref{eq:com2}). Since $\unf^*= \invunf$ by (\ref{eq:inverse}), we have $C_s^*\mathcal{T}_{-\epsilon}=(\mathcal{T}_\epsilon C_s)^*\supseteq (C_s\mathcal{T}_\epsilon)^*=\mathcal{T}_{-\epsilon}C_s^*$. Hence, we deduce for all $\epsilon\in\mathbb{R}\setminus\{0\}$
  \[
     \mathcal{T}_{\epsilon} C_s^*C_s \subseteq  C_s^*\mathcal{T}_{\epsilon} C_s\subseteq C_s^*C_s \mathcal{T}_{\epsilon}.
  \]
  Since $C_s^*C_s$ is self-adjoint and $\mathcal{T}_\varepsilon$ continuous, we deduce that for all bounded, measurable functions $f\colon \sigma(C_s^*C_s)\to \mathbb{R}$, $\mathcal{T}_{\epsilon} f(C_s^*C_s)=f(C_s^*C_s)\mathcal{T}_\epsilon$. In particular, $f=\chi_{\{0\}}$ is a possible choice.  $\chi_{\{0\}}(C_s^*C_s)$ is the orthogonal projection onto $\kar(C_s^*C_s)$. Note that $\kar(C_s^*C_s)=\kar(C_s)$. Indeed, $\kar(C_s)\subseteq \kar(C_s^*C_s)$ is trivial; for $\phi\in \kar(C_s^*C_s)$, we test the equation $C_s^*C_s\phi=0$ with $\phi\in \kar(C_s^*C_s)\subseteq\dom(C_s^*C_s)\cap \dom (C_s)$ to obtain $\langle C_s\phi,C_s\phi\rangle = 0$, that is, $\kar(C_s)\supseteq \kar(C_s^*C_s)$.  We infer $P=\chi_{\{0\}}(C_s^*C_s)$, which yields the assertion.
\end{proof}

\begin{lemma}\label{lem:proj} Let $K\subseteq H_s$ be a closed subspace and $P$ the projection onto $K$. Then $PC_d\subseteq C_dP$.
\end{lemma}
\begin{proof}
 We have identified $C_d$ with $C_d\otimes 1_{H_s}$ and $P$ with $1_{H_d}\otimes P$. Thus, the assertion follows from standard tensor-product theory of operators in Hilbert spaces, see, e.g., \cite[Appendix]{W11_P}.
\end{proof}

\begin{lemma}\label{thm:bas2} Let $(u_\epsilon)_\epsilon$ and $(\epsilon C_d u_\varepsilon)_\epsilon$ be bounded in $(H_d\otimes H_s)^m$ and $(H_d\otimes H_s)^n$, respectively. Then there are subsequences of $(u_\epsilon)_\epsilon$ and $(\epsilon C_d u_\varepsilon)_\epsilon$ as well as $u\in \dom(C_s)$ such that $(u_\epsilon)_\epsilon$ weakly 2-scale converges to $u$ and $(\epsilon C_d u_\varepsilon)_\epsilon$ weakly 2-scale converges to $C_su$.
\end{lemma}
\begin{proof}
There exist subsequences such that $(\mathcal{T}_\varepsilon u_\epsilon)_\epsilon$ weakly  converges to $u\in (H_d\otimes H_s)^m$ and $(\epsilon \mathcal{T}_\epsilon C_d u_\varepsilon)_\epsilon$ weakly converges to $v\in (H_d\otimes H_s)^n$, respectively. We will show that $u\in \dom(C_s)$ and $C_s u=v$. For this, we let $g\in \dom(C_s^*)\cap \dom(C_d^*)$ and compute (using (\ref{eq:com1b}))
 \begin{align*}
   \langle \epsilon \mathcal{T}_\epsilon C_d u_\epsilon, g\rangle & =
      \langle   u_\epsilon,\epsilon C_d^*\mathcal{T}_{-\epsilon} g\rangle \\
      & = \langle   u_\epsilon,\epsilon \mathcal{T}_{-\epsilon} C_d^*g+\mathcal{T}_{-\epsilon} C_s^*g\rangle 
   \\ & = \langle \epsilon  \mathcal{T}_\epsilon u_\epsilon, C_d^*g\rangle + \langle  \mathcal{T}_\epsilon u_\epsilon, C_s^*g\rangle
   \\ & \to \langle u,C_s^*g\rangle.
 \end{align*}
Since $\dom(C_s^*)\cap \dom(C_d^*)$ is an operator core for $C_s^*$ by \eqref{eq:core}, we obtain $u\in \dom(C_s)$ and $C_su=(\textrm{w-})\lim_{\epsilon\to 0}\epsilon \mathcal{T}_\epsilon C_d u_\epsilon$. 
\end{proof}

\begin{theorem}[Compactness]\label{prop:wo3} Assume \eqref{eq:core}--\eqref{eq:com2}.

(a) Let $(u_\epsilon)_\epsilon$ be uniformly bounded in $\dom(C_d)$. Assume that $(u_\epsilon)_\epsilon$ and $(C_du_\epsilon)_\epsilon$ weakly 2-scale converge. Then there exist $u\in \dom(C_d)\cap \kar(C_s)$ and $v\in \overline{\rge}(C_s)$ such that $u_\epsilon\stackrel{2}{\rightharpoonup} u$ and $C_d u_\epsilon \stackrel{2}{\rightharpoonup} C_d u+v$.

(b) Let $(q_\epsilon)_\epsilon$ be uniformly bounded in $\dom(C_d^*)$. Assume that $(q_\epsilon)_\epsilon$ and $(C_d^*q_\epsilon)_\epsilon$ weakly 2-scale converge. Then there exist $q\in \dom(C_d^*)\cap \kar(C_s^*)$ and $w\in \overline{\rge}(C_s^*)$ such that $q_\epsilon\stackrel{2}{\rightharpoonup} q$ and $C_d^* q_\epsilon \stackrel{2}{\rightharpoonup} C_d^* q+w$.

(c) Assume, in addition, $m=1$ and that \eqref{eq:inv} holds. Let $(u_\varepsilon)_\epsilon$ be uniformly bounded in $\dom(C_d)$. Then there exists $u\in\dom(C_d)\cap \kar(C_s)$ and $v\in \overline{\rge}(C_s)\cap \overline{\rge}(C_s^*)^n\subseteq (H_d\otimes H_s)^n$ such that (a subsequence of) $(u_\varepsilon)_\epsilon$ weakly 2-scale converges to $u$ and $(C_d u_\varepsilon)_\epsilon$ weakly 2-scale converges to $C_du+v$. 
\end{theorem}
\begin{proof}
(a) For suitable subsequences, denote by $u$ and $\overline{v}$ the corresponding weak 2-scale-limits of $(u_\epsilon)_\epsilon$ and $(C_du_\epsilon)_\epsilon$. By Lemma \ref{thm:bas2}, we deduce that $u\in \kar(C_s)$. By weak continuity of $P$, it is easy to see that $(P\mathcal{T}_\epsilon u_\epsilon)_\epsilon$ weakly converges to $Pu=u$. By Lemma \ref{lem:TPPT}, we have $P\mathcal{T}_\epsilon=\mathcal{T}_\epsilon P$. Thus, since $PC_d\subseteq C_dP$, we get $P u_\epsilon \in \dom(C_d)\cap \kar(C_s)\subseteq \dom(C_d)\cap \dom(C_s)$ and so, by \eqref{eq:com1},
 \[
   C_d P \mathcal{T}_\epsilon u_\epsilon = \mathcal{T}_\epsilon C_d Pu_\epsilon = \mathcal{T}_\epsilon PC_d u_\epsilon.
 \]
  Thus, since the right-hand side is uniformly bounded in $\epsilon>0$, so is the left-hand side. From the closedness of $C_dP$, it follows that $u=Pu\in \dom(C_d)$.
  
  We define $v\coloneqq \overline{v}-C_d u$. Then we compute for suitable $\epsilon>0$ and all $w\in \kar(C_s^*)\cap \dom(C_d^*)$:
\begin{equation*}
   \langle \mathcal{T}_\epsilon C_d u_\epsilon,w\rangle = \langle  \mathcal{T}_\epsilon u_\epsilon,C_d^* w\rangle.
\end{equation*}
 Letting $\epsilon\to 0$ on both sides, we obtain
 \[
    \langle C_d u + v, w\rangle = \langle  C_d u, w\rangle,
 \]
 which leads to 
 \[
    \langle v,w\rangle = 0 \quad(w\in \kar(C_s^*)\cap \dom(C_d^*)).
 \]
 Since $\kar(C_s^*)\cap \dom(C_d^*)$ is dense in $\kar(C_s^*)$ by \eqref{eq:core}, we infer $v\in \kar(C_s^*)^{\bot}=\overline{\rge}(C_s)$.

(b) By symmetry of the conditions \eqref{eq:core}--\eqref{eq:com2} in $C_d$ and $C_d^*$, the proof follows analogously to (a).

(c) Choose subsequences of $(\mathcal{T}_\epsilon u_\epsilon)_{\epsilon}$ and $(\mathcal{T}_\epsilon C_du_\epsilon)_\epsilon$ that weakly converge to some $u$ and $\overline{v}$. By Lemma \ref{thm:bas2}, we deduce that $C_s u = (\textrm{w-})\lim_{\epsilon\to 0}\epsilon \mathcal{T}_\epsilon C_d u_\epsilon = 0\cdot \overline{v}$, which yields $u\in\kar(C_s)$. Moreover, we may choose a weakly convergent subsequence of $(u_\epsilon)_{\epsilon}$ in $\dom (C_d)$. Denote the limit by $\tilde u$. Let $P$ be the orthogonal projection onto $\kar(C_s)$. As $m=1$, $\kar(C_s)\subseteq H_s$ and Lemma \ref{lem:proj} yields $PC_d\subseteq C_d P$. In particular, we obtain that $(Pu_\epsilon)_\epsilon$ weakly converges to $P\tilde u$ in $\dom(C_d)$. Furthermore, by $\mathcal{T}_\epsilon v=v$ on $\kar(C_s)$ (see \eqref{eq:inv}), we infer $\mathcal{T}_\epsilon P = P$. Moreover, from Lemma \ref{lem:TPPT}, we have $\mathcal{T}_\epsilon P = P\mathcal{T}_\epsilon$. Thus, we get that
\[
   u = P\textrm{w-}\lim_{\epsilon\to 0} \mathcal{T}_\epsilon u_\epsilon=\textrm{w-}\lim_{\epsilon\to 0}P \mathcal{T}_\epsilon u_\epsilon = \textrm{w-}\lim_{\epsilon\to 0} \mathcal{T}_\epsilon P u_\epsilon = \textrm{w-}\lim_{\epsilon\to 0} P u_\epsilon = P\tilde u\in \dom(C_d).
\]
We define $v \coloneqq \overline{v}-C_d u$. Then we compute for suitable $\epsilon>0$ and all $w\in \kar(C_s^*)\cap \dom(C_d^*)$:
\begin{equation*}
   \langle \mathcal{T}_\epsilon C_d u_\epsilon,w\rangle  = \langle u_\epsilon, \mathcal{T}_{-\epsilon} C_d^*w+ \frac{1}{\epsilon}\mathcal{T}_{-\epsilon}C_s^* w\rangle = \langle  \mathcal{T}_\epsilon u_\epsilon,C_d^* w\rangle.
\end{equation*}
 Letting $\epsilon\to 0$ on both sides, we obtain
 \[
    \langle C_d u + v, w\rangle = \langle  C_d u, w\rangle,
 \]
 which leads to 
 \[
    \langle v,w\rangle = 0 \quad(w\in \kar(C_s^*)\cap \dom(C_d^*)).
 \]
 Since $\kar(C_s^*)\cap \dom(C_d^*)$ is dense in $\kar(C_s^*)$ by \eqref{eq:core}, we infer $v\in \kar(C_s^*)^{\bot}=\overline{\rge}(C_s)$.
 
 Finally, we let $w\in \kar(C_s)\cap \dom(C_d^*)$. Then, using $\mathcal{T}_\epsilon v=v$ for all $v\in \kar(C_s)$, we have for suitable $\epsilon>0$
 \begin{equation*}
   \langle \mathcal{T}_\epsilon C_d u_\epsilon, w\rangle  = \langle C_du_\epsilon,w\rangle = \langle u_\epsilon, C_d^*w\rangle = \langle u_\epsilon, \mathcal{T}_{-\epsilon}C_d^*w\rangle =\langle  \mathcal{T}_{\epsilon}u_\epsilon, C_d^*w\rangle.
 \end{equation*}
 Letting $\epsilon\to 0$, we get 
 \[
    \langle C_d u + v, w\rangle = \langle u,C_d^*w\rangle = \langle C_d u,w\rangle.
 \]
 Hence, since $\dom(C_d^*)\cap \kar(C_s)$ is dense in $\kar(C_s)$ as $m=1$, we deduce
 \[
   \langle v,w\rangle = 0 \quad(w\in\kar(C_s)),
 \]
 which leads to $v\in \overline{\rge}(C_s^*)$.
\end{proof}
\section{Homogenization of (abstract) elliptic problems}\label{sec:4}

In order to illustrate our so far findings, we shall treat an elliptic homogenization problem. Note that this is the abstract variant of the classical result \cite[Theorem 4.1.1]{Bourgeat1994} (see Section \ref{sec:elliptic} for the periodic case). We assume throughout that $m=1$ and that \eqref{eq:core}--\eqref{eq:inv}.

\begin{theorem}\label{thm:ellhom} Let $f\in \kar(C_s)$ and assume that $\rge(C_d)$ is closed and $C_d$ is injective. Let $A \in L((H_d\otimes H_s)^n)$ satisfy $\Re A = \frac{1}{2}(A+A^*)\geq c$ for some $c>0$. For $\epsilon>0$ consider 
\begin{equation}\label{eq:ellhom}
   C_d^* \mathcal{T}_{-\epsilon}A\mathcal{T}_{\epsilon}C_d u_\epsilon = f.
\end{equation}
Then $(u_\epsilon)_\epsilon$ is well-defined and uniformly bounded in $\dom(C_d)$ and strongly 2-scale converges to some $u\in \dom(C_d)\cap \kar(C_s)$. Moreover, $(C_du_\epsilon)_\epsilon$ weakly 2-scale converges to $C_d u+v$ for some $v\in \overline{\rge}(C_s)\cap \overline{\rge}(C_s^*)^n$, where $u,v$ are the unique solutions of the following system of equations
\begin{align}\label{eq:ellhom2a}
   C_d^* PA(C_du+v) &= f,\\ \label{eq:ellhom2b}
   C_s^*A(C_d u + v) & = 0.
\end{align}
Here, $P$ is the orthogonal projection onto $\kar(C_s)$.
\end{theorem}

Before we come to a proof of Theorem \ref{thm:ellhom}, we address well-posedness of \eqref{eq:ellhom}. For this, we will use the direct approach outlined in \cite[Theorem 3.1]{TW14_FE}; see also \cite[Theorem 2.9]{W18_NHC}.

\begin{theorem}\label{thm:wpeps}
 Assume the conditions of Theorem \ref{thm:ellhom} to be in effect. Then for all $\epsilon>0$ there exists a unique $u_\epsilon\in \dom(C_d)$ satisfying \eqref{eq:ellhom}. Moreover, we have that $(u_\epsilon)_\epsilon$ is uniformly bounded in $\dom(C_d)$. 
\end{theorem}
\begin{proof} Let $\iota\colon \rge(C_d) \hookrightarrow (H_d\otimes H_s)^n$ denote the canonical embedding. Then $\iota^*$ is the orthogonal projection onto $\rge(C_d)$, see \cite[Lemma 3.2]{PTW15_FI}. Moreover, from $\kar(C_d^*)=\rge(C_d)^\bot$, it is easy to see that
 \[
    C_d^* \mathcal{T}_{-\epsilon}A\mathcal{T}_{\epsilon}C_d = C_d^*\iota\iota^* \mathcal{T}_{-\epsilon}A\mathcal{T}_{\epsilon}\iota\iota^*C_d. 
 \]
 By the closed graph theorem $\iota^*C_d$ is continuously invertible; by \cite[Lemma 2.4 and Corollary 2.5]{TW14_FE} $(\iota^*C_d)^*=C_d^*\iota$ is also continuously invertible. Next, from $\Re A\geq c$ it follows that $\Re \mathcal{T}_{-\epsilon}A\mathcal{T}_{\epsilon}\geq c$. In consequence, we obtain $\Re \iota^*\mathcal{T}_{-\epsilon}A\mathcal{T}_{\epsilon}\iota\geq c\iota^*\iota$. Thus, using that $\rge(C_d^*)$ is closed and dense in $H_d\otimes H_s$ by the injectivity of $C_d$, we obtain
 \begin{equation}\label{eq:wp1}
    u_\epsilon = (\iota^*C_d)^{-1} \left(\iota^* \mathcal{T}_{-\epsilon}A\mathcal{T}_{\epsilon}\iota\right)^{-1}(C_d^*\iota)^{-1} f,
 \end{equation}
 which yields uniqueness of solutions of \eqref{eq:ellhom}. Multiplying this equality by $\iota^*C_d$, we infer
 \begin{equation}\label{eq:wp2}
      \iota^*C_du_\epsilon = \left(\iota^* \mathcal{T}_{-\epsilon}A\mathcal{T}_{\epsilon}\iota\right)^{-1}(C_d^*\iota)^{-1} f.
 \end{equation}
 The equalities \eqref{eq:wp1} and \eqref{eq:wp2} together with 
 \[
    \|\left(\iota^* \mathcal{T}_{-\epsilon}A\mathcal{T}_{\epsilon}\iota\right)^{-1}\|\leq \frac{1}{c}
 \]
yield that $(u_\epsilon)_\epsilon$ is uniformly bounded in $\dom(C_d)$.
\end{proof}

The next result settles uniqueness of the homogenized equations stated in Theorem \ref{thm:ellhom}. The rationale is similar to the one in \cite[Theorem 4.1.1]{Bourgeat1994}; however we do not need to impose the curl-condition nor do we use any variant of `Kozlov's identity' (see \cite[Lemma 2.4 and the subsequent remark]{Bourgeat1994}).

\begin{lemma}\label{lem:unique} Let $u\in \dom(C_d)\cap \kar(C_s)$ and $v\in \overline{\rge}(C_s)\cap \overline{\rge}(C_s^*)^n$ satisfy \eqref{eq:ellhom2a} and \eqref{eq:ellhom2b} with $f=0$. Assume that $C_d$ is injective. Then $u=0$ and $v=0$.  
\end{lemma}
\begin{proof} We define $\zeta\coloneqq C_d u+v$. Since $\zeta-C_du=v \in \overline{\rge}(C_s)$ and $A(C_d u+v)\in \kar(C_s^*)$, by \eqref{eq:ellhom2b},
\[
    0 = \langle A\zeta, v\rangle = \langle A \zeta,\zeta-C_d u\rangle.
\]
Thus, we deduce from \eqref{eq:ellhom2a} with $f=0$ using $u\in \dom(C_d)\cap \kar(C_s)$
\[
   \langle A\zeta, \zeta\rangle = \langle A\zeta, C_d u\rangle=\langle A\zeta, C_d P u\rangle=\langle PA\zeta, C_d u\rangle=\langle C_d^*PA\zeta, u\rangle=0.
\]
Thus, by $\Re A\geq c$, we infer $\zeta = 0$. Since $\kar(C_s)^n\ni C_d u \bot v \in \overline{\rge}(C_s^*)^n$, we obtain the assertion.
\end{proof}

The following simple lemma will be useful in the proof of Theorem \ref{thm:ellhom}. It provides a recovery construction for the weak two-scale limit of the sequence $C_d u_{\varepsilon}$.
\begin{lemma}\label{lemma:570}
Let $u\in \dom(C_d)\cap \kar(C_s)$ and $v\in \overline{\rge}(C_s)$. For $\delta>0$, there exists $\varphi_{\delta} \in \dom(C_d)\cap \dom(C_s)$ such that
\begin{equation}\label{eq:571}
\| C_s \varphi_{\delta} - v\|\leq \delta.
\end{equation} 
Moreover, for $\varphi_{\delta,\varepsilon}:= \invunf u + \varepsilon \invunf \varphi_{\delta}$, we obtain 
\begin{equation*}
\unf \varphi_{\delta,\varepsilon} \to u, \quad \unf C_{d} \varphi_{\delta,\varepsilon} \to C_d u + C_s \varphi_{\delta} \quad (\text{as }\varepsilon\to 0).
\end{equation*}
\begin{proof}
By the definition of $\overline{\rge}(C_s)$ and using \eqref{eq:core}, we obtain that there exists $\varphi_{\delta} \in \dom(C_d)\cap \dom(C_s)$ which satisfies (\ref{eq:571}). 
Also, since $\unf \varphi_{\delta,\varepsilon}= u+ \varepsilon \varphi_{\delta}$, it follows that $\unf \varphi_{\delta,\varepsilon}\to u$. Furthermore, using (\ref{eq:com1})-(\ref{eq:com2}) we compute
\begin{equation*}
C_d \varphi_{\delta,\varepsilon} = \invunf C_d u + \varepsilon \invunf C_d \varphi_{\delta} + \invunf C_s \varphi_{\delta}.
\end{equation*}
This implies that $\unf C_d \varphi_{\delta,\varepsilon}\to C_d u + C_s \varphi_{\delta}$.
\end{proof}
\end{lemma}

\begin{proof}[Proof of Theorem \ref{thm:ellhom}] The family $(u_\epsilon)_{\epsilon>0}$ is well-defined and bounded in $\dom(C_d)$ by Theorem \ref{thm:wpeps}. By Theorem \ref{prop:wo3} (c), we find a subsequence of $(u_\epsilon)_{\epsilon>0}$, $u\in \dom(C_d)\cap \kar(C_s)$, $v\in \overline{\rge}(C_s^*)^n\cap \overline{\rge}(C_s)$ such that $(u_\epsilon)_{\epsilon}$ and $(C_d u_\epsilon)_{\epsilon}$ weakly 2-scale converge to $u$ and $C_d u+v$, respectively. Let $\phi\in \dom(C_d)$ and $\psi\in \dom(C_s)\cap \dom(C_d)$. Then we compute for suitable $\epsilon>0$ using \eqref{eq:ellhom}
\begin{align*}
  \langle f,\phi\rangle + \langle \epsilon f, \mathcal{T}_{-\epsilon}\psi\rangle & = \langle \mathcal{T}_{-\epsilon} A \mathcal{T}_\epsilon C_d u_\epsilon,  C_d P \phi+ \epsilon C_d \mathcal{T}_{-\epsilon}\psi\rangle
  \\ & = \langle \mathcal{T}_{-\epsilon} A \mathcal{T}_{\epsilon} C_d u_\epsilon, C_d P \phi + \epsilon \mathcal{T}_{-\epsilon} C_d\psi + C_s \mathcal{T}_{-\epsilon}\psi\rangle
  \\ & = \langle  A \mathcal{T}_{\epsilon} C_d u_\epsilon, \mathcal{T}_{\epsilon} C_d P \phi\rangle +
  \langle  A \mathcal{T}_{\epsilon} C_d u_\epsilon, \epsilon  C_d\psi\rangle +
  \langle  A \mathcal{T}_{\epsilon} C_d u_\epsilon, C_s \psi\rangle,
\end{align*}
  where we used $\epsilon C_d \mathcal{T}_{-\epsilon}\psi=\epsilon \mathcal{T}_{-\epsilon} C_d\psi + C_s \mathcal{T}_{-\epsilon}\psi$ and $\mathcal{T}_{-\epsilon}C_s\psi=C_s\mathcal{T}_{-\epsilon}\psi$. Since $C_dP\phi\in \kar(C_s)$, we have $\mathcal{T}_{\epsilon} C_d P \phi= \mathcal{T}_{\epsilon}P C_d  \phi = PC_d\phi$. Thus, we obtain
 \[
    \langle f,\phi\rangle + \langle \epsilon f, \mathcal{T}_{-\epsilon}\psi\rangle=\langle  A \mathcal{T}_{\epsilon} C_d u_\epsilon, P C_d \phi\rangle + \langle  A \mathcal{T}_{\epsilon} C_d u_\epsilon, \epsilon  C_d\psi\rangle +
  \langle  A \mathcal{T}_{\epsilon} C_d u_\epsilon, C_s \psi\rangle.
 \]
In this equality we may let $\epsilon\to 0$ and obtain the asserted equalities \eqref{eq:ellhom2a} and \eqref{eq:ellhom2b}. As the solutions to \eqref{eq:ellhom2a} and \eqref{eq:ellhom2b} are unique by Lemma \ref{lem:unique}, we deduce that any subsequence of $(u_\epsilon)_\epsilon$ and $(C_du_\epsilon)_\epsilon$ weakly 2-scale converges to $u$ and $C_d u+v$, which implies that both $(u_\epsilon)_\epsilon$ and $(C_du_\epsilon)_\epsilon$ weakly 2-scale converge without the choice of subsequences.

In order to show strong convergence of $(\unf u_{\varepsilon})_\epsilon$, we consider  $\varphi_{\delta}$ and $\varphi_{\delta,\varepsilon}$ defined as in Lemma \ref{lemma:570} (for $u,v$ the solution of (\ref{eq:ellhom2a})-(\ref{eq:ellhom2b})).
Since $C_d$ is injective and it has closed range, it follows (see, e.g., \cite[Remark 3.2(b)]{TW14_FE}) that there exists $c_0>0$ such that $\|\varphi \|\leq c_0 \|C_d \varphi\|$ for all $\varphi \in \dom(C_d)$. As a result of this, we obtain
\begin{equation}\label{eq:605}
\|u_{\varepsilon}-\varphi_{\delta,\varepsilon}\| \leq c_0 \|C_{d}(u_{\varepsilon}-\varphi_{\delta,\varepsilon})\| = c_0\|\unf C_d (u_{\varepsilon}-\varphi_{\delta,\varepsilon})\|.
\end{equation}
Using the assumption $\Re A\geq c$ (in the inequality) and that $u_{\varepsilon}$ solves (\ref{eq:ellhom}) (in the equality), we obtain
\begin{align*}
c \|\unf C_d (u_{\varepsilon}- \varphi_{\delta,\varepsilon})\|^2 & \leq \Re \expect{A \unf C_d( u_{\varepsilon} -\varphi_{\delta,\varepsilon} ),\unf C_d( u_{\varepsilon} -\varphi_{\delta,\varepsilon})} \\ & = \Re \expect{f, u_{\varepsilon} - \varphi_{\delta,\varepsilon}} - \Re \expect{A \unf C_d \varphi_{\delta,\varepsilon},\unf C_d(u_{\varepsilon}-\varphi_{\delta,\varepsilon})}.
\end{align*}
Since $f \in \kar (\unf-1)$ (for any $\varepsilon \neq 0$), we conclude that the first term on the right-hand side equals $\Re \expect{f, \unf u_{\varepsilon}-\unf \varphi_{\delta,\varepsilon}}$. As a result of this and using the properties of $\varphi_{\delta,\varepsilon}$ (Lemma \ref{lemma:570}), we obtain that for fixed $\delta>0$,
\begin{equation*}
\limsup_{\varepsilon\to 0} 
\|\unf C_d (u_{\varepsilon}- \varphi_{\delta,\varepsilon})\|^2 \leq  -\Re\expect{A\brac{C_d u + C_s \varphi_{\delta}},v-C_s \varphi_{\delta}}.
\end{equation*}
Moreover, using the above and (\ref{eq:605}) we obtain (using that the second term vanishes in the limit $\varepsilon\to 0$ by Lemma \ref{lemma:570}) for some $c_1\geq 0$ and all $\delta>0$
\begin{equation*}
\limsup_{\varepsilon\to 0} \brac{\|u_{\varepsilon}-\varphi_{\delta,\varepsilon}\|^2+ \|\unf \varphi_{\delta,\varepsilon}- u\|^2+ \|\unf C_d(u_{\varepsilon}-\varphi_{\delta,\varepsilon})\|^2}\leq c_1 \|C_d u + C_s\varphi_{\delta}\| \|v-C_s \varphi_{\delta}\|.
\end{equation*}
By the choice of $\varphi_{\delta}$ (Lemma \ref{lemma:570}), in the limit $\delta \to 0$ the right-hand side vanishes. Consequently, we find a diagonal sequence $\delta(\varepsilon)\to 0$ (as $\varepsilon\to 0$) such that 
\begin{equation}\label{eq:622}
\|u_{\varepsilon}-\varphi_{\delta(\varepsilon),\varepsilon}\|+ \|\unf \varphi_{\delta(\varepsilon),\varepsilon}- u\|+ \|\unf C_d(u_{\varepsilon}-\varphi_{\delta(\varepsilon),\varepsilon})\|\to 0 \quad \text{as }\varepsilon\to 0.
\end{equation}
Finally, this yields
\begin{equation*}
\|\unf u_{\varepsilon} - u\| \leq \| \unf u_{\varepsilon}-\unf \varphi_{\delta(\varepsilon),\varepsilon}\|+ \|\unf \varphi_{\delta(\varepsilon),\varepsilon}- u\| \to 0 \quad \text{as }\varepsilon\to 0.\qedhere
\end{equation*}
\end{proof}
The above proof implies the following abstract corrector type result (see (\ref{eq:622})).
\begin{corollary} 
Assume the conditions of Theorem \ref{thm:ellhom} to be in effect and let $u_{\varepsilon},u, v$ be given as in Theorem \ref{thm:ellhom}. There exists a sequence $\varphi_{\varepsilon} \in \dom(C_d)\cap \dom(C_s)$ such that
\begin{equation*}
\varepsilon \varphi_{\varepsilon}\to 0, \quad 
\unf C_d(u_{\varepsilon}-\varepsilon \mathcal{T}_{-\varepsilon}\varphi_{\varepsilon}) \to C_d u \quad \text{strongly in }H_d \otimes H_s.
\end{equation*}
\end{corollary}
\section{Homogenization of abstract evolutionary equations}
\subsection{A Hilbert space framework for evolutionary equations}\label{sec:ahil}

For the application to time-dependent problems in the subsequent sections, we shall specify the particular framework, we are working in. For this, we recall some results from \cite{PicPhy}, where the setting was introduced for the first time. We shall also refer to \cite[Section 2]{KPSTW14_OD} for more details, when properties of the time-derivative to be introduced in the following are concerned.  

To begin with, we define for $\nu\in\mathbb{R}_{\ge0}$ the space $L_\nu^2(\mathbb{R};H)$ of (equivalence classes of) Hilbert space $H$-valued functions $f$ being Bochner measurable and satisfy
\[
   \|f\|_{\nu}^2\coloneqq \int_\mathbb{R} \|f(t)\|_H^2 \exp(-2\nu t) dt<\infty.
\]
$L_\nu^2(\mathbb{R};H)$ is a Hilbert space under the norm $\|\cdot\|_\nu$. Denoting by $H_\nu^1(\mathbb{R};H)$ the (first) Sobolev space of weakly differentiable functions with distributional derivative representable as an element in $L_\nu^2(\mathbb{R};H)$, we define the \emph{time-derivative} as the operator
\[
  \partial_{t,\nu} \colon H_\nu^1(\mathbb{R};H)\subseteq L_\nu^2(\mathbb{R};H)\to L_\nu^2(\mathbb{R};H), f\mapsto f'.
\]
It turns out that $\partial_{t,\nu}$ is a normal operator. More particularly, $\partial_{t,\nu}$ admits a spectral representation as multiplication operator in $L^2(\mathbb{R};H)$. For this, we define the multiplication-by-the-argument operator
\[
   \m \colon \dom(\m)\subseteq L^2(\mathbb{R};H) \to L^2(\mathbb{R};H), \phi\mapsto (\xi\mapsto \xi\phi(\xi)),
\]
where
\[
 \dom(\m)\coloneqq \{ \phi\in L^2(\mathbb{R};H); (\xi\mapsto \xi \phi(\xi))\in L^2(\mathbb{R};H)\}.
\]
Define the ($H$-valued) Fourier transformation $\mathcal{F}\colon L^2(\mathbb{R};H)\to L^2(\mathbb{R};H)$ as the unitary extension of 
\[
   \mathcal{F} f\coloneqq \big(\xi \mapsto \frac{1}{\sqrt{2\pi}}\int_\mathbb{R} \exp(-\i t\xi)f(t)dt\big)\quad (f\in L^1(\mathbb{R};H)\cap L^2(\mathbb{R};H)).
\]
Moreover, we define for $\nu\in\mathbb{R}_{\ge 0}$ the (obviously) unitary operator
\[
   \exp(-\nu \m)\colon L_\nu^2(\mathbb{R};H)\to L^2(\mathbb{R};H), f\mapsto (t\mapsto \e^{-\nu t} f(t)),
\]
and we set
\[
   \mathcal{L}_\nu \coloneqq \mathcal{F}\exp(-\nu \m),
\]
the \emph{Fourier--Laplace transformation}. With the latter transformation at hand the explicit spectral theorem for $\partial_{t,\nu}$ (see also \cite[Corollary 2.5(c)]{KPSTW14_OD}) reads
\[
   \partial_{t,\nu} = \mathcal{L}_\nu^* \big(\i \m +\nu\big) \mathcal{L}_\nu.
\]
The latter equation yields a functional calculus for $\partial_{t,\nu}$ (or for $\partial_{t,\nu}^{-1}$.). In fact, let $\nu_0\in\mathbb{R}_{\ge 0}$ and $\nu>\nu_0$ and let $\mathcal{M}\in \mathcal{H}(\mathbb{C}_{\Re>\nu_0};L(H,K))$, where $K$ is another Hilbert space and
\[
   \mathcal{H}(\mathbb{C}_{\Re>\nu_0};L(H,K))\coloneqq \{\mathcal{M}\colon \mathbb{C}_{\Re>\nu_0}\to L(H,K); \mathcal{M} \text{ analytic}\}. 
\]
Then we define
\[
   \mathcal{M}\big(\partial_{t,\nu}\big)\coloneqq \mathcal{L}_\nu^* \left(\mathcal{M}\big(\i \m +\nu\big)\right) \mathcal{L}_\nu,
\]
where 
\[
   \left(\mathcal{M}\big(\i \m +\nu\big)\phi\right)(\xi)\coloneqq \mathcal{M}\big(\i \xi +\nu\big)\phi(\xi)\quad(\xi\in\mathbb{R})
\]
for all compactly supported, continuous functions $\phi\colon \mathbb{R}\to H$.

The reason we focus on analytic functions $\mathcal{M}$ rather than continuous or even just measurable functions is that analyticity of $\mathcal{M}$ and causality of $\mathcal{M}(\partial_{t,\nu})$ are strongly related, which is apparent from the Paley--Wiener theorem, see e.g., \cite[Section 2]{PicPhy} or \cite[Theorem 2.4]{PTW15_WP_P}, \cite[Section 1.2]{W16_H}.

We are now in the position to formulate the well-posedness theorem, which can be viewed as underlying structure of many linear equations in mathematical physics and continuum mechanics.

\begin{theorem}[{{\cite[Solution Theory]{PicPhy}}}]\label{thm:st} Let $\nu_0\in\mathbb{R}_{\ge 0}$, $\nu>\nu_0$, $H$ Hilbert space, $A\colon \dom(A)\subseteq H\to H$ a skew-self-adjoint operator, i.e., $A=-A^*$, and $\mathcal{M}\in \mathcal{H}(\mathbb{C}_{\Re>\nu_0};L(H))$. Assume there exists $c>0$ such that
\[
    \Re \langle \mathcal{M}(z)\phi,\phi\rangle \geq c\langle \phi,\phi\rangle\quad(z\in\mathbb{C}_{\Re>\nu_0},\phi\in H).
\]
Then the operator $B\coloneqq \mathcal{M}(\partial_{t,\nu})+A$ is densely defined on $L_\nu^2(\mathbb{R};H)$. Moreover, $B$ is closable and $\overline{B}$ is continuously invertible in $L_\nu^2(\mathbb{R};H)$, $\|\overline{B}^{-1}\|\leq 1/c$ and $\overline{B}^{-1}$ is \emph{causal}, that is,
\[
   \mathbf{1}_{\mathbb{R}_{\leq a}}\overline{B}^{-1}\mathbf{1}_{\mathbb{R}_{\leq a}}=\mathbf{1}_{\mathbb{R}_{\leq a}}\overline{B}^{-1}\quad(a\in\mathbb{R}).
\] 
\end{theorem}
\begin{proof}
 For $z\in \mathbb{C}_{\Re>\nu_0}$ it follows from \cite[Lemma 2.12]{EGW17_D2N} that $\mathcal{M}(z)+A$ is continuously invertible in $H$ with inverse satisfying $\|(\mathcal{M}(z)+A)^{-1}\|\leq 1/c$. Thus, \cite[Remark 2.3(a)]{Trostorff2015a} applies and we obtain the assertion.
\end{proof}

We define \[\mathcal{H}^\infty(\mathbb{C}_{\Re>\nu_0};L(H,K))\coloneqq \{\mathcal{M}\in \mathcal{H}(\mathbb{C}_{\Re>\nu_0};L(H,K)); \mathcal{M} \text{ bounded}\},\] which turns into a Banach space if endowed with the supremum norm.

\begin{proposition}\label{prop:hinfty} Assume the hypotheses of Theorem~\ref{thm:st}. Then 
\[
   \Big[\mathbb{C}_{\Re>\nu_0} \ni z \mapsto \big(\mathcal{M}(z)+A\big)^{-1} \in L(H)\Big] \in \mathcal{H}^\infty(\mathbb{C}_{\Re>\nu};L(H)).
\] 
\end{proposition}
\begin{proof}
 The claim follows from \cite[Lemma 2.12]{EGW17_D2N} and the fact that composition of analytic mappings are analytic again.
\end{proof}

It will be the next theorem, which forms the basic result for the convergence results to be followed in the next section.

\begin{theorem}\label{thm:conv} Let $\nu_0\in\mathbb{R}_{\ge 0}$, $H,K$ separable Hilbert spaces. Let $(S_\epsilon)_{\epsilon>0}$ be a bounded family in $\mathcal{H}^\infty(\mathbb{C}_{\Re>\nu_0};L(H,K))$ and let $S_0\in \mathcal{H}^\infty(\mathbb{C}_{\Re>\nu_0};L(H,K))$. Assume that for all $z\in \mathbb{C}_{\Re>\nu_0}$ we have
\[
   S_\epsilon (z) \to S_0(z)\quad(\epsilon \to 0)
\]
in the weak operator topology of $L(H,K)$.

Then $S_\epsilon(\partial_{t,\nu})\to S_0(\partial_{t,\nu})$ in the weak operator topology of $L(L_\nu^2(\mathbb{R};H),L_\nu^2(\mathbb{R};K))$. 
\end{theorem}

For the proof of Theorem~\ref{thm:conv}, we shall use the following result.

\begin{theorem}[{{\cite[Theorem 4.3]{W14_FE} and \cite[Lemma 3.5]{W12_HO}}}]\label{thm:com} Let $\Omega\subseteq \mathbb{C}$ open, $H,K$ separable Hilbert spaces. Let $(S_\epsilon)_\epsilon$ be a bounded family in $\mathcal{H}^\infty(\mathbb{C}_{\Re>\nu_0};L(H,K))$. Then there exists \[T\in \mathcal{H}^\infty(\mathbb{C}_{\Re>\nu_0};L(H,K))\] and a nullsequence $(\epsilon_k)_{k\in\mathbb{N}}$ in $(0,\infty)$ satisfying for all $K\subseteq \Omega$ compact
\[
   S_{\epsilon_k}(z) \to T(z)\quad(k\to\infty, z\in K)
\]
in the weak operator topology of $L(H,K)$. Moreover, if $\Omega=\mathbb{C}_{\Re>\nu_0}$, then
\[
   S_{\epsilon_k}(\partial_{t,\nu})\to T(\partial_{t,\nu})\quad (k\to\infty)
\]
in the weak operator topology of $L(L_\nu^2(\mathbb{R};H),L_\nu^2(\mathbb{R};K))$ for all $\nu>\nu_0$. 
\end{theorem}

\begin{proof}[Proof of Theorem \ref{thm:conv}] Choose $(\epsilon_k)_k$ and $T$ according to Theorem~\ref{thm:com}. Then, for $z\in\mathbb{C}_{\Re>\nu_0}$, we obtain
\[
   S_0 (z) = (\tau_{\textnormal{w}}\text{-})\lim_{k\to\infty} S_{\epsilon_k}(z) =  T(z).
\]
Moreover, we get $S_{\epsilon_k}(\partial_{t,\nu})\to S_0(\partial_{t,\nu})$ in the weak operator topology of $L(L_\nu^2(\mathbb{R};H),L_\nu^2(\mathbb{R};K))$. The subsequence principle concludes the proof. 
\end{proof}

\subsection{Applications to time-dependent-type problems}\label{sec:appl}

We shall treat a subclass of evolutionary equations discussed in the previous section. The particular cases treated here cover the heat equation, the wave equation, the Maxwell's equations and even systems of mixed type formulated on possibly rough domains, which do not need to satisfy any boundedness conditions, as we shall demonstrate in the subsequent sections by means of examples.

More specifically, in this section, we confine ourselves to the following class of problems. Define $H_0\coloneqq (H_d \otimes H_s)^m$, $H_1\coloneqq (H_d\otimes H_s)^n$ and $H=H_0\oplus H_1$. Let $\nu_0\geq 0$,  $\mathcal{M}_k\colon \mathbb{C}_{\Re>\nu_0}\to L(H_k)$ be analytic and assume that $\Re \mathcal{M}_k(z)\geq c$ for all $z\in\mathbb{C}_{\Re>\nu_0}$, $k\in\{0,1\}$ and some $c>0$. We need to restrict ourselves to a certain class of right-hand sides. For this we set
\begin{align*}
   \mathcal{H}_0\coloneqq \kar(C_s)\cap \bigcap_{\epsilon\in \mathbb{R}\setminus\{0\}}\kar(\mathcal{T}_\epsilon - 1) \subseteq H_0, \quad    \mathcal{H}_1\coloneqq \kar(C_s^*)\cap \bigcap_{\epsilon\in \mathbb{R}\setminus\{0\}}\kar(\mathcal{T}_\epsilon - 1) \subseteq H_1.
\end{align*}
We remark that in the applications, e.g., to stochastic homogenization (see Sections \ref{sec:maxwell} and \ref{sec:wave}), the above choice allows the consideration of deterministic right-hand sides (that is a standard assumption in stochastic homogenization), i.e., $L^2(Q)\otimes \C\subseteq \mathcal{H}_i$.  

For $(f,g)\in \mathcal{H}:=\mathcal{H}_0\oplus \mathcal{H}_1$, we consider for $\epsilon>0$
\begin{equation}\label{eq:tdp}
  \left(\mathcal{T}_{-\epsilon}\begin{pmatrix} \mathcal{M}_0(z) & 0\\ 0 & \mathcal{M}_1(z)\end{pmatrix}\mathcal{T}_{\epsilon}+\begin{pmatrix} 0 & C_d^*\\ -C_d & 0\end{pmatrix}\right)\begin{pmatrix} u_\epsilon\\ q_\epsilon \end{pmatrix} = \begin{pmatrix} f \\ g \end{pmatrix}.
\end{equation}
Note that the operator 
\[
   B_\epsilon\coloneqq\overline{\left(\mathcal{T}_{-\epsilon}\begin{pmatrix} \mathcal{M}_0(\partial_{t,\nu}) & 0\\ 0 & \mathcal{M}_1(\partial_{t,\nu})\end{pmatrix}\mathcal{T}_{\epsilon}+\begin{pmatrix} 0 & C_d^*\\ -C_d & 0\end{pmatrix}\right)}
\]
is continuously invertible in $L_\nu^2(\mathbb{R};H)$ by Theorem~\ref{thm:st} applied to $\mathcal{M}=\mathcal{T}_{-\epsilon}\diag(\mathcal{M}_0,\mathcal{M}_1)\mathcal{T}_{\epsilon}$ and $A=\left(\begin{smallmatrix} 0 & C_d^*\\ -C_d & 0\end{smallmatrix}\right)$. We shall also introduce
\begin{equation}\label{eq:Sez}
   S_\epsilon (z) \coloneqq \overline{\left(\mathcal{T}_{-\epsilon}\begin{pmatrix} \mathcal{M}_0(z) & 0\\ 0 & \mathcal{M}_1(z)\end{pmatrix}\mathcal{T}_{\epsilon}+\begin{pmatrix} 0 & C_d^*\\ -C_d & 0\end{pmatrix}\right)}^{-1}
\end{equation}
for all $\epsilon>0$ and $z\in\mathbb{C}_{\Re>\nu_0}$. By Proposition \ref{prop:hinfty}, we have that $(S_\epsilon)_{\epsilon>0}$ is a bounded family in $\mathcal{H}^\infty(\mathbb{C}_{\Re>\nu_0};L(H))$.

Our aim in this section will be to construct an operator-valued function $S_0\colon z\mapsto S_0(z)$ such that
\[
  \mathcal{T}_\epsilon S_\epsilon (z) \to S_0(z)\quad(\epsilon\to 0)
\]
in the weak operator topology (in an appropriate space). It will turn out that $S_0(z)$ can be written as $S_0(z)=\left(\tilde M(z)+\left(\begin{smallmatrix} 0 & C_d^*\\ -C_d & 0\end{smallmatrix}\right)\right)^{-1}$ for suitable $\tilde{M}(z)$ to be described explicitly below. Finally, we shall conclude with an application of Theorem~\ref{thm:conv} to obtain a homogenization result for the full time-dependent problem.

We will suppress the dependence of $u_\epsilon$ and $q_\epsilon$ on $z$ for the time being; at the end of this section we come back to this.

\begin{remark}
 We refer to \cite{A11,PTW15_WP_P,W16_H} and the references therein for an instance of the many examples that are covered by this equation. We will consider some special cases in the next section. The rather involved homogenization result for a suitable class of non-diagonal $M$ is treated in \cite{W16_HPDE}. In that paper, however, a compactness assumption had to be introduced, which we do not assume here. Moreover, in \cite{W16_HPDE}
 the local problem is given implicitly and there is no criterion ensuring convergence without the extraction of subsequences. However, in the framework of so-called `nonlocal $H$-convergence' a convergence result for Maxwell's equations was shown in \cite{W18_NHC}. For a setting strictly confined to periodic problems defined on the whole Euclidean space as underlying domain, quantitative results can be found in \cite{CW17_FH}.
\end{remark}

First of all we establish existence and boundedness of $(u_\epsilon,q_\epsilon)_\epsilon$ in $\dom(C_d)\oplus \dom(C_d^*)$. The result is as plain as it is to establish the uniform bound in $H$. We provide some more details as follows.

\begin{proposition}\label{prop:wp} For all $\epsilon>0$, $(u_\epsilon,q_\epsilon)$ is well-defined as a solution of \eqref{eq:tdp}. Moreover, the family $(u_\epsilon,q_\epsilon)_{\epsilon}$ is uniformly bounded in $\dom(C_d)\oplus \dom(C_d^*)$. 
\end{proposition}
\begin{proof}
 Note that $\begin{pmatrix} 0 & C_d^*\\ -C_d & 0\end{pmatrix}$ is skew-self-adjoint and $\Re \mathcal{T}_{-\epsilon}\mathcal{M}(z)\mathcal{T}_{\epsilon}\geq c$ for all $z\in\mathbb{C}_{\Re>\nu}$, $\epsilon>0$ and some $c>0$. Hence, the assertion follows from \cite[Lemma 2.12]{EGW17_D2N} applied to $M=\mathcal{T}_{-\epsilon}\mathcal{M}(z)\mathcal{T}_{\epsilon}$ and $A=\begin{pmatrix} 0 & C_d^*\\ -C_d & 0\end{pmatrix}$.
\end{proof}

The main result of this section is presented next, that is, we will now present the main step to establish convergence of $(\mathcal{T}_\epsilon S_\epsilon(z))_\epsilon$. We shall show convergence of a subsequence first. Then, we will prove uniqueness of the limit, so that the following theorem actually also holds without the choice of subsequences. Below, $P_{\kar{(C_s)}}$ and $P_{\kar{(C_s^*)}}$ denote the orthogonal projections to $\kar{(C_s)}$ and $\kar{(C_s^*)}$.

\begin{theorem}\label{thm:th} Assume \eqref{eq:core}--\eqref{eq:com2}. For $\epsilon>0$ let $(u_\epsilon,q_\epsilon)_{\epsilon}$ be given by \eqref{eq:tdp}. Then (a subsequence of) $(u_\epsilon)_\epsilon$, $(q_\epsilon)_{\epsilon}$ weakly 2-scale converge to some $u\in \dom(C_d)\cap \kar(C_s)$, $q\in \dom(C_d^*)\cap \kar(C_s^*)$. $(u,q)$ satisfies the following system of equations
 \begin{align}
  P_{\kar{(C_s)}}\mathcal{M}_0(z) u + P_{\kar{(C_s)}}C_d^*q &=f  \label{eq:th1}
  \\ P_{\kar{(C_s^*)}} \mathcal{M}_1(z)q - P_{\kar{(C_s^*)}}C_du & = g. \label{eq:th2}
 \end{align}
\end{theorem}
\begin{proof}
  By Proposition \ref{prop:wp}, $(u_\epsilon)_\epsilon$ and $(q_\epsilon)_\epsilon$ are bounded in $\dom(C_d)$ and $\dom(C_d^*)$, respectively. Hence, by Theorem \ref{prop:wo3} (a) and (b), we find subsequences of $(u_\epsilon)_\epsilon$ and $(q_\epsilon)_\epsilon$ as well as $u\in \dom(C_d)\cap \kar(C_s)$, $v\in \overline{\rge}(C_s)$, $q\in \dom(C_d^*)\cap \kar(C_s^*)$ and $w\in \overline{\rge}(C_s^*)$ such that $(u_\epsilon)_\epsilon$, $(q_\epsilon)_\epsilon$, $(C_du_\epsilon)_\epsilon$, and $(C_d^*q_\epsilon)_\epsilon$ weakly 2-scale converge to $u$, $q$, $C_du+v$, and $C_d^*q+w$, respectively. Next, using \eqref{eq:tdp}, we have for all $\phi \in (H_d\otimes H_s)^m$ and suitable $\epsilon>0$
  \begin{align*}
     \langle f,\phi\rangle & =\langle \mathcal{T}_{\epsilon}  f, \phi\rangle
     \\ & = \langle P_{\kar{(C_s)}} f,\mathcal{T}_{-\epsilon}\phi\rangle
     \\ & =\langle \mathcal{T}_{-\epsilon}\mathcal{M}_0(z)\mathcal{T}_{\epsilon} u_\epsilon+C_d^*q_\epsilon,\mathcal{T}_{-\epsilon}P_{\kar{(C_s)}}\phi\rangle 
     \\ & = \langle \mathcal{M}_0(z)\mathcal{T}_{\epsilon} u_\epsilon, P_{\kar{(C_s)}}\phi\rangle+\langle \mathcal{T}_{\epsilon} C_d^*q_\epsilon,P_{\kar{(C_s)}}\phi\rangle
     \\ & \to \langle \mathcal{M}_0(z)u, P_{\kar{(C_s)}}\phi\rangle+\langle C_d^*q+w,P_{\kar{(C_s)}}\phi\rangle
     \\ & = \langle P_{\kar{(C_s)}}\mathcal{M}_0(z)u, \phi\rangle+\langle P_{\kar{(C_s)}} C_d^*q, \phi\rangle \quad(\epsilon\to 0).
  \end{align*}
  We obtain
  \[
     P_{\kar{(C_s)}}\mathcal{M}_0(z) u + P_{\kar{(C_s)}}C_d^*q = f,
  \]
  which yields \eqref{eq:th1}. \eqref{eq:th2} follows analogously. 
\end{proof}

The next result is a reformulation of the system \eqref{eq:th1}-\eqref{eq:th2}. For this, we introduce the canonical embeddings
\begin{align*}
   \iota_s &\colon \kar(C_s)  \hookrightarrow (H_d\otimes H_s)^m,\\
   \iota_{s^*}& \colon \kar(C_s^*)\hookrightarrow (H_d\otimes H_s)^n.
\end{align*}
Since $u=\iota_s \iota_s^* u$ and $q=\iota_{s^*} \iota_{s^*}^* q$, we obtain the following. 
\begin{corollary}\label{cor:sys} Let $(u,q)$ satisfy the system \eqref{eq:th1}-\eqref{eq:th2}. Then 
\begin{align}
  \iota_s^* \mathcal{M}_0(z) \iota_s \iota_s^* u + \iota_s^* C_d^* \iota_{s^*}\iota_{s^*}^*q & = \iota_s^*f, \label{eq:sys1} \\
  \iota_{s^*}^* M_1(z) \iota_{s^*}\iota_{s^*}^* q - \iota_{s^*}^* C_d \iota_s \iota_{s}^* u & = \iota_{s^*}^*g. \label{eq:sys2}
\end{align} 
\end{corollary}

\begin{remark}\label{rem:rep} The equations for $u$ and $q$ from Corollary \ref{cor:sys} can be written in the following block operator matrix form
\[
   \left(\begin{pmatrix}
      \widetilde{\mathcal{M}}_0(z) & 0 \\ 0 & \widetilde{\mathcal{M}}_1(z)
   \end{pmatrix} + \begin{pmatrix} 0 & \iota_s^*C_d^*\iota_{s^*} \\ -\iota_{s^*}^*C_d\iota_s & 0\end{pmatrix}\right)\begin{pmatrix}
  \tilde u \\ \tilde q  \end{pmatrix} = \begin{pmatrix} \tilde f\\ \tilde g \end{pmatrix},
\]
where $\widetilde{\mathcal{M}}_0(z)=\iota_s^* \mathcal{M}_0(z) \iota_s$, $\widetilde{\mathcal{M}}_1(z)=\iota_{s^*}^*\mathcal{M}_1(z)\iota_{s^*}$, $\tilde u= \iota_s^* u$, $\tilde q= \iota_{s^*}^* q$, $\tilde f= \iota_s^* f$ and $\tilde g = \iota_{s^*}^* g$. 
\end{remark}

From this remark we obtain the \emph{homogenized} problem as follows. Namely,
\[
   \left(\begin{pmatrix}
      \widetilde{\mathcal{M}}_0(z) & 0 \\ 0 & \widetilde{\mathcal{M}}_1(z)
   \end{pmatrix} + \begin{pmatrix} 0 & \widetilde{C_d}^* \\ -\widetilde{C_d} & 0\end{pmatrix}\right)\begin{pmatrix}
  w \\ r  \end{pmatrix} = \begin{pmatrix} h\\ j\end{pmatrix},
\]
where $\widetilde{C_d}=\overline{\iota_{s^*}^*C_d\iota_s}$ and $\widetilde{C_d}^*=\overline{\iota_s^*C_d^*\iota_{s^*}}$. In particular, we deduce that $(\tilde{u},\tilde{q})$ is uniquely determined, see \cite[Lemma 2.12]{EGW17_D2N}.

This observation combined with Corollary~\ref{cor:sys} yields that the claim of Theorem~\ref{thm:th} is true even \emph{without} choosing subsequences. In any case, we can formulate the following homogenization result for time-dependent homogenization problems, which is one of the main results of this article. We define $\mathcal{S}_\epsilon(z) \in L(\mathcal{H},H)$ by
 \[
      \mathcal{S}_\epsilon(z) (f,g)\coloneqq {S}_\epsilon(z)\begin{pmatrix} f\\ g \end{pmatrix},
 \]
 where $S_\epsilon(z)$ is given by \eqref{eq:Sez}. We shall also define $\mathcal{S}_0(z)\in L(\mathcal{H},H)$ via  \[
    \mathcal{S}_0(z) (f,g)\coloneqq  \begin{pmatrix}
    \iota_s \; &  0 \\ 0 \; & \iota_{s^*}
    \end{pmatrix} \left(\begin{pmatrix}
      \widetilde{\mathcal{M}}_0(z) & 0 \\ 0 & \widetilde{\mathcal{M}}_1(z)
   \end{pmatrix} + \begin{pmatrix} 0 & \widetilde{C_d}^* \\ -\widetilde{C_d} & 0\end{pmatrix}\right)^{-1}\begin{pmatrix} \iota_{s}^* \; & 0 \\ 0 \; & \iota_{s^*}^*  \end{pmatrix}\begin{pmatrix}
  f \\ g  \end{pmatrix}.
 \]
\begin{theorem}\label{thm:mr} Assume \eqref{eq:core}--\eqref{eq:com2} and that $H$ is separable. Then we have
\[
    \mathcal{T}_\epsilon\mathcal{S}_\epsilon(\partial_{t,\nu})\to \mathcal{S}_0(\partial_{t,\nu})\quad(\epsilon\to 0)
\]
 in the weak operator topology of $L(L_\nu^2(\mathbb{R};\mathcal{H}),L_\nu^2(\mathbb{R};H))$.
\end{theorem}
\begin{proof} For $z\in\mathbb{C}_{\Re>\nu_0}$ we obtain with Theorem~\ref{thm:th} (in the version without choosing subsequences, which is justified by the subsequence principle in combination with Corollary~\ref{cor:sys} and the representation in Remark~\ref{rem:rep}) that $\mathcal{T}_\epsilon\mathcal{S}_\epsilon(z)\to \mathcal{S}_0(z)$ in the weak operator topology of $L(\mathcal{H},H)$. Thus, the assertion follows from Theorem~\ref{thm:conv}.
\end{proof}

\section{Examples}\label{sec:exam}
In this section we present two specific examples of the unfolding operator---the stochastic and (a variant of the) periodic unfolding operator. Moreover, we provide specific examples in which Theorem \ref{thm:ellhom} and Theorem \ref{thm:mr} yield homogenization results. In particular, we consider homogenization problems for elliptic, Maxwell's and wave equations. Also, besides the essential homogenization results obtained by Theorem \ref{thm:mr} for the evolutionary equations, with little additional effort we prove some corrector type results.

\subsection{Examples of unfolding operators}
\textbf{Deterministic differential operators.}
First, we introduce the deterministic differential operators which we use in the description of the considered problems. In the following, we denote by $\partial_i$, $\nabla=(\partial_1,\ldots,\partial_n)$, and $\dive\cdot$ the partial derivative, the gradient and the divergence, respectively, and use the notation
$$\nabla\times u:=
\brac{\partial_{2} u_3 - \partial_{3} u_2,\ \partial_{3} u_1 - \partial_{1} u_3,\ \partial_{1} u_2 - \partial_{2} u_1 }$$
for the curl of a smooth three-dimensional vector-field. Let $Q\subseteq \R^n$ be open. 

\textit{Gradient and divergence operators}. The operator closure of $C^{\infty}_c(Q)\subseteq L^2(Q)\to L^2(Q)^n$, $\varphi\mapsto\nabla\varphi$ and its adjoint are denoted by
\begin{align*}
  \nablao: \dom(\nablao)=H^1_0(Q)\subseteq L^2(Q)\to L^2(Q)^n\qquad\text{and}\qquad 
  -\dive_x = \brac{\nablao}^{*}.
\end{align*}

\textit{Curl operator}. Let $n=3$. The operator closure of $C_c^{\infty}(Q)^3\subseteq L^2(Q)^3\to L^2(Q)^3$, $u\mapsto\nabla\times u$ and its adjoint are denoted by
\begin{align*}
  \curlo: \dom(\curlo)\subseteq L^2(Q)^3\to L^2(Q)^3\qquad\text{and}\qquad \curl_x:=\brac{\curlo}^*.
\end{align*}

\subsubsection{Periodic unfolding (in the mean)} We present a periodic unfolding operator suited for homogenization problems involving periodic coefficients. We consider the unit torus $\Box\coloneqq \R^n / \Z^n$ equipped with the push-forward of the Lebesgue measure $\mathcal{L}([0,1)^n)$ and let $Q \subseteq \R^n$ be open. We consider the following particular choice of the abstract Hilbert spaces from Section \ref{s:gr}: $H_d = L^2(Q), \; H_s = L^2(\Box)$. For $\varepsilon \in \R \setminus{\cb{0}}$, we let $\mathcal{T}_{\varepsilon}:L^2(Q) \overset{a}{\otimes} L^2(\Box)\to L^2(Q)\otimes L^2(\Box)$ be given by
\begin{equation}\label{equation:869}
\mathcal{T}_{\varepsilon}u(x,y)= u\brac{x, y - \cb{\frac{x}{\varepsilon}}_{\Box}},
\end{equation}
where $\cb{\cdot}_{\Box}: \R^{n}\to \Box$ is the canonical quotient map. $\unf$ is a linear isometry and we call its continuous extension $\unf: L^2(Q)\otimes L^2(\Box)\to L^2(Q)\otimes L^2(\Box)$ \textit{periodic unfolding operator}. Thus, $\unf$ is unitary and we have $\mathcal{T}_{\varepsilon}^{-1}=\mathcal{T}_{-\varepsilon}$.
\begin{remark}
The above notion of periodic unfolding differs from the classical notion of unfolding introduced in \cite{cioranescu2002periodic,cioranescu2008periodic}. In particular, in \cite{cioranescu2002periodic,cioranescu2008periodic} the unfolding operator is defined as
\begin{equation*}
\widetilde{\unf}: L^2(Q)\to L^2(Q)\otimes L^2(\Box), \text{ given by }\widetilde{\unf} u(x,y)=u\brac{\left[\frac{x}{\varepsilon}\right]_{\varepsilon \Z^n}+\varepsilon y},
\end{equation*}
where $\left[{\cdot}\right]_{\varepsilon \Z^n}: \R^n\to \varepsilon \Z^n$ is defined as $\left[ x \right]_{\varepsilon \Z^n}=x-\varepsilon \cb{\frac{x}{\varepsilon}}_{\Box}$. The operator $\widetilde{\unf}$ defined in this fashion is not surjective and therefore it does not satisfy the assumptions of our abstract setting. On the other hand, definition (\ref{equation:869}) involves an additional variable $y$ for functions in the domain of $\unf$ (in this respect we might call the notion defined in (\ref{equation:869}) \textit{periodic unfolding in the mean}) and consequently it fulfills the requirements from Section \ref{s:gr}.
\end{remark}
The operator closure of $C^{\infty}(\Box)\subseteq L^2(\Box)\to L^2(\Box)^n$, $\varphi\mapsto\nabla\varphi$ and its adjoint are denoted by 
\begin{align*}
 \nablap: \dom(\nablap)=H^1_{\#}(\Box)\subseteq L^2(\Box)\to L^2(\Box)^n\qquad\text{and}\qquad- \divp = (\nablap)^*.
\end{align*}
Note that we may identify $C^\infty(\Box)$ with the space of $\Z^n$-periodic functions in $C^\infty(\R^n)$, and $L^2(\Box)$ with $L^2((0,1)^n)$.

If we set $C_d = \nablao$ and $C_s = \mathrm{grad}_y^{\#}$, it follows that the abstract conditions (\ref{eq:com1})-(\ref{eq:inv}) are satisfied (Note that \eqref{eq:core} is easy, by choosing $D_1=D_2=C_c^\infty(Q)$ and $E_1=E_2= C^{\infty}(\Box)$):
\begin{lemma}\label{lemma:890}
For $\varepsilon \neq 0$, we have
\begin{align*}
\varepsilon \nablao \mathcal{T}_{-\varepsilon}\varphi = \varepsilon \mathcal{T}_{-\varepsilon} \nablao \varphi + \mathrm{grad}_y^{\#} \mathcal{T}_{-\varepsilon}\varphi, & \quad \text{for all }\varphi\in \dom(\nablao)\cap \dom(\mathrm{grad}_y^{\#}), \\
\unf \mathrm{grad}_y^{\#} \subseteq \mathrm{grad}_y^{\#} \mathcal{T}_{\varepsilon}, & \\
\unf \varphi = \varphi, & \quad \text{for all }\varphi \in \kar(\mathrm{grad}_y^{\#}).
\end{align*}
\end{lemma}
Note that  $\dom(\nablao)\cap \dom(\mathrm{grad}_y^{\#})=H^1_0(Q)\otimes L^2(\Box)\cap L^2(Q)\otimes H^1_{\#}(\Box)$ and $\kar(\mathrm{grad}_y^{\#})\simeq L^2(Q)\otimes \mathbb{C}$.
\begin{proof}[Proof of Lemma \ref{lemma:890}]
Let $\eta \in C_c^{\infty}(Q)\overset{a}{\otimes} C^{\infty}(\Box)^n$. Then
\begin{equation*}
\expect{ \varepsilon \mathcal{T}_{-\varepsilon} \nablao \varphi, \eta }_{L^2(Q)\otimes L^2(\Box)^n} = -\expect{\varphi, \varepsilon \dive_x \unf \eta}_{L^{2}(Q)\otimes L^2(\Box)}.
\end{equation*}
Using the smoothness of $\eta$ and the chain rule, we compute $\varepsilon \dive_{x} \unf \eta =\varepsilon \unf  \dive_{x} \eta - \unf \dive_{y}^{\#} \eta$, where we use that $\dive_x \unf \eta (x,y)=\partial_{x_i} \brac{\eta_i(x,y-\cb{\frac{x}{\varepsilon}}_{\Box})}$ a.e. Therefore, we obtain \begin{equation*}
\expect{ \varepsilon \mathcal{T}_{-\varepsilon} \nablao \varphi, \eta }_{L^2(Q)\otimes L^2(\Box)^n}=\expect{ \varepsilon \nablao \invunf \varphi - \mathrm{grad}_y^{\#} \invunf \varphi, \eta }_{L^2(Q)\otimes L^2(\Box)^n}.
\end{equation*}
Since $C_c^{\infty}(Q)\overset{a}{\otimes} C^{\infty}(\Box)$ is dense in $\dom(\nablao)\cap \dom(\mathrm{grad}_y^{\#})$, the first claim follows. 

The second claim follows from the fact that $\unf \mathrm{grad}_y^{\#} \varphi = \mathrm{grad}_y^{\#} \unf \varphi$ for any $\varphi\in C^{\infty}(\Box)$.

The last claim follows using that $\kar(\mathrm{grad}_y^{\#})\simeq L^2(Q)\otimes \C$, which can be obtained with the help of the Poincar{\'e} inequality: There exists $C>0$ such that
\begin{equation*}
\left\|\varphi - \int_{\Box}\varphi  \right\|_{L^2(\Box)} \leq C \|\mathrm{grad}_y^{\#} \varphi\|_{L^2(\Box)^n} \quad \text{for all }\varphi \in \dom(\mathrm{grad}_y^{\#}). \qedhere
\end{equation*}
\end{proof}

\subsubsection{Stochastic unfolding}\label{sec:ex2} Let $Q\subseteq \R^n$ be open, $(\Omega,\Sigma, \mu, \tau)$ be a probability space satisfying Assumption \ref{assumpt:267}. In this section, we set $H_d= L^2(Q)$ and $H_s= L^2(\Omega)$ and consider the \textit{stochastic unfolding operator} $\unf: L^2(Q){\otimes}L^2(\Omega)\to L^2(Q)\otimes L^2(\Omega)$ defined in (\ref{eq:298}).
\begin{remark}
The stochastic unfolding operator is a generalization of the periodic unfolding operator since in the case $(\Omega,\Sigma,\mu)  = (\Box,\mathcal{L},dy)$ and $\tau_x y = y + \cb{x}_{\Box}$ the two operators coincide. Another instance of stochastic unfolding corresponds to the choice $(\Omega,\Sigma,\mu)=\brac{\Box^m, \mathcal{L}^m,(dy)^m}$ (for some $m \in \N$) with $\tau_x y=(y_1+\cb{x}_{\Box},\ldots,y_m+\cb{x}_{\Box})$. The latter is well-suited for the treatment of problems involving quasi-periodic rapidly-oscillating coefficients.
\end{remark}

With help of the dynamical system $\tau$ (for fixed $i\in\{1,\ldots,n\}$) we introduce the group of unitary operators ($h\in \R$) $T_{h e_i}: L^2(\Omega) \to L^2(\Omega)$ given by
\begin{equation}\label{eq:988}
T_{h e_i} \varphi = \varphi \circ \tau_{h e_i}.
\end{equation}
This group is strongly continuous and we denote its infinitesimal generator by $\partial_{\omega,i}: \dom(\partial_{\omega,i})\subseteq L^2(\Omega)\to L^2(\Omega)$, i.e., 
\begin{equation*}
\partial_{\omega,i} \varphi = \lim_{h \to 0} \frac{T_{h e_i}\varphi-\varphi}{h}.
\end{equation*}
Analogously to \eqref{eq:988}, we define $T_x:L^2(\Omega)\to L^2(\Omega)$ for $x\in \R^n$. Using the \textit{stochastic partial derivatives} $\partial_{\omega,i}$, we define the \textit{stochastic gradient} \[\mathrm{grad}_{\omega}:H^1(\Omega)\coloneqq \bigcap_{i\in\{1,\ldots,n\}}\dom(\partial_{\omega,i})\subseteq L^2(\Omega) \to L^2(\Omega)^n\] given by
\begin{equation*}
\mathrm{grad}_{\omega} \varphi = \brac{\partial_{\omega,1}\varphi,\ldots, \partial_{\omega,n}\varphi}.
\end{equation*}
It is a closed and densely defined linear operator and we let $\dive_{\omega}= - \brac{\mathrm{grad}_{\omega}}^*$. Also, we introduce the space of \textit{smooth} random variables by
\begin{equation*} 
H^{\infty}(\Omega)=\cb{\varphi \in L^2(\Omega): \; \partial_{\omega,1}^{\alpha_1}\cdots\partial_{\omega,n}^{\alpha_n}\varphi \in H^1(\Omega)\; \text{for all }\alpha_1,\ldots,\alpha_n \in \N_0}.
\end{equation*}

In the case $n=3$, we define the stochastic counterpart of the $\curlo$ operator as follows. Let $\nabla_{\omega}\times:  H^{\infty}(\Omega)^3 \subseteq L^2(\Omega)^3 \to L^2(\Omega)^3$ be given by
\begin{equation*}
\nabla_{\omega}\times \varphi = \brac{\partial_{\omega,2} \varphi_3 - \partial_{\omega,3} \varphi_2,\partial_{\omega,3} \varphi_1 - \partial_{\omega,1} \varphi_3,\partial_{\omega,1} \varphi_2 - \partial_{\omega,2} \varphi_1 }.
\end{equation*}
We let $\curl_{\omega}= \brac{\nabla_{\omega}\times}^{**}$. 
The choice $C_d= \nablao$ and $C_s= \mathrm{grad}_{\omega}$ satisfies assumptions (\ref{eq:com1})-(\ref{eq:inv}), and if we set $C_d= \curlo$ and $C_s=\curl_{\omega}$ we merely obtain (\ref{eq:com1})-(\ref{eq:com2}). (Again, note that \eqref{eq:core} in both cases follows with the choice $D_1=D_2=C_c^\infty(Q)$ and $E_1=E_2=H^\infty(\Omega)$.)
\begin{lemma}\label{lemma:964}
Let $\varepsilon\neq 0$. 
\begin{enumerate}
\item Then
\begin{align*}
\varepsilon \nablao \mathcal{T}_{-\varepsilon}\varphi = \varepsilon \mathcal{T}_{-\varepsilon} \nablao \varphi + \mathrm{grad}_{\omega} \mathcal{T}_{-\varepsilon}\varphi & \quad \text{for all }\varphi \in \dom(\nablao)\cap \dom{(\mathrm{grad}_{\omega})},\\
\unf \mathrm{grad}_{\omega}\subseteq \mathrm{grad}_{\omega} \unf &,\\
\unf \varphi = \varphi & \quad \text{for all }\varphi \in \kar(\mathrm{grad}_{\omega}).
\end{align*}
\item If $n=3$, then
\begin{align*}
\varepsilon \curlo \mathcal{T}_{-\varepsilon} \varphi = \varepsilon \mathcal{T}_{-\varepsilon} \curlo \varphi + \curl_{\omega} \mathcal{T}_{-\varepsilon} \varphi & \quad \text{for all } \varphi \in \dom(\curlo)\cap \dom(\curl_{\omega}),\\ \unf \curl_{\omega} \subseteq \curl_{\omega} \unf &.
\end{align*}
\end{enumerate}
\end{lemma}
\begin{proof}
\textit{(a)} Let $\eta \in C^{\infty}_{c}(Q)\overset{a}{\otimes} H^{\infty}(\Omega)^n$. We have
\begin{equation*}
\expect{\varepsilon \invunf \nablao \varphi,\eta}_{L^2(Q)\otimes L^2(\Omega)}=\expect{\varphi,\varepsilon \dive_x \unf \eta}_{L^2(Q)\otimes L^2(\Omega)}.
\end{equation*}
We compute $\varepsilon \dive_x \unf \eta = \varepsilon \unf \dive_{x} \eta - \unf \dive_{\omega} \eta$, in order to obtain,
\begin{equation*}
\expect{\varepsilon \invunf \nablao \varphi,\eta}_{L^2(Q)\otimes L^2(\Omega)}=\expect{\varepsilon \nablao \invunf \varphi-\mathrm{grad}_{\omega}\invunf \varphi,\eta}_{L^2(Q)\otimes L^2(\Omega)}.
\end{equation*}
Note that $C^{\infty}_{c}(Q)\overset{a}{\otimes} H^{\infty}(\Omega)$ is dense in $\dom(\nablao)\cap \dom{(\mathrm{grad}_{\omega})}$ (see (\ref{eq:992}) below) and therefore the first claim follows.

In order to obtain the second claim, it is sufficient to show that $\unf \partial_{\omega,i} \varphi = \partial_{\omega,i} \unf \varphi$ for any $\varphi \in H^{\infty}(\Omega)$ (for any $i\in\{1,\ldots,n\}$). Let $\varphi, \eta \in L^2(Q)\stackrel{a}\otimes H^{\infty}(\Omega)$. We have
\begin{align*}
\expect{\unf \partial_{\omega,i} \varphi, \eta}_{L^2(Q) \otimes L^2(\Omega)} &= \lim_{h\to 0}\frac{1}{h} \expect{\unf T_{h e_i}\varphi - \unf \varphi, \eta}_{L^2(Q)\otimes L^2(\Omega)} \\ & =  \lim_{h\to 0}\frac{1}{h}\expect{\unf \varphi,T_{-h e_i}\eta - \eta}_{L^2(Q)\otimes L^2(\Omega)} \\ & = -\expect{\unf \varphi,\partial_{\omega,i}\eta}_{L^2(Q)\otimes L^2(\Omega)}\\ & = \expect{\partial_{\omega,i} \unf \varphi, \eta}_{L^2(Q)\otimes L^2(\Omega)}.
\end{align*} 
As a result of this and by the density of $H^{\infty}(\Omega)$ in $H^1(\Omega)$, the claim follows.

For $\varphi \in \kar(\mathrm{grad}_{\omega})$, we have $T_{x}\varphi =\varphi$ for all $x\in \R^n$ (see, e.g., \cite[Lemma 3.10]{heida2019stochastic}). As a result of this, the third claim follows.

\textit{(b)} The proof follows analogously to part \textit{(a)}.
\end{proof}

\begin{remark}[Boundary conditions]
The first commutation relation from Lemmas \ref{lemma:890} and \ref{lemma:964} (a) remains valid if the gradient operator $\nablao$ (which is defined on a domain which accounts for homogeneous Dirichlet boundary conditions) is replaced by gradient operators corresponding to other types of boundary conditions (e.g., with homogeneous Neumann condition or periodic boundary condition). In this respect, the unfolding procedure does not only apply to problems with certain boundary conditions.
\end{remark}

\textbf{Auxiliary results.} We provide certain facts that will be helpful in the treatment of the stochastic homogenization problems considered in the following section (in particular for corrector type results). The following standard orthogonal decompositions hold (see, e.g., \cite{brezis2010functional})
\begin{equation}\label{equation:957}
L^2(\Omega)=\kar(\mathrm{grad}_{\omega}) \oplus^{\bot} \overline{\rge}(\dive_{\omega}), \quad
L^2(\Omega)^n = \kar(\dive_{\omega})\oplus^{\bot} \overline{\rge}(\mathrm{grad}_{\omega}).
\end{equation}
$P_{\mathsf{inv}}$ and $P_{\mathsf{pot}}$ denote the orthogonal projections $P_{\mathsf{inv}}: L^2(\Omega) \to \kar(\mathrm{grad}_{\omega})$ and $P_{\mathsf{pot}}:L^2(\Omega)^n\to \overline{\rge}(\mathrm{grad}_{\omega})=: L^2_{\mathsf{pot}}(\Omega)$. We have $\kar(\mathrm{grad}_{\omega})= \cb{u \in L^2(\Omega): T_x u = u \quad \text{for all } x\in \R^n}=: L^2_{\mathsf{inv}}(\Omega)$ (see, e.g., \cite[Lemma 3.10]{heida2019stochastic}). If we additionally assume that the probability space is ergodic, we get $L^2_{\mathsf{inv}}(\Omega)\simeq \C$.

\begin{lemma}\label{lemma:981}
Let $n=3$. Then
\begin{equation*}
\kar( \curl_{\omega}) = L^2_{\mathsf{inv}}(\Omega)^3 \oplus^{\bot} L^2_{\mathsf{pot}}(\Omega), \quad
\kar(\curl_{\omega}^{*})= L^2_{\mathsf{inv}}(\Omega)^3 \oplus^{\bot} L^2_{\mathsf{pot}}(\Omega).
\end{equation*}
\end{lemma}
The following standard mollification procedure for random variables (see, e.g., \cite{jikov2012homogenization,Bourgeat1994}) is useful in the proof of the above lemma. For a sequence $\delta \to 0$, we consider a sequence of standard mollifiers $\rho_{\delta}\in C^{\infty}_{c}(\R^n)$ (even and non-negative) and for $\varphi\in L^2(\Omega)$, we set
\begin{equation}\label{eq:992}
\varphi_{\delta}= \int_{\R^n} \rho_{\delta}(y)T_{y}\varphi dy.
\end{equation}
We obtain that $\varphi_{\delta}\in H^{\infty}(\Omega)$ (apply the definition of $\partial_{\omega,i}$, transform the integral to have the difference quotient on $\rho_\delta$, and then apply the dominated convergence theorem), $\expect{\varphi_{\delta}}=\expect{\varphi}=\int_\Omega \phi(\omega)d\mu(\omega)$ and $\varphi_{\delta}\to \varphi$ in $L^2(\Omega)$ as $\delta\to 0$ (see, e.g., \cite[Section 7.2]{jikov2012homogenization}). We collect some further useful properties of this mollification procedure: 
\begin{lemma}\label{lemma:996} Let $\phi\in L^2(\Omega)$. Then $\phi_\delta\in H^\infty(\Omega)$; and $\phi_\delta \to\phi$ in $L^2(\Omega)$ as $\delta\to 0$. 

Moreover, the following statements hold:
\begin{enumerate}
\item If $\varphi \in H^1(\Omega)$, then $\mathrm{grad}_{\omega}\varphi_{\delta}=\brac{\mathrm{grad}_{\omega}\varphi}_{\delta}$ and $\mathrm{grad}_{\omega}\varphi_{\delta}\to \mathrm{grad}_{\omega}\varphi$. 
\item If $\varphi \in \dom(\dive_{\omega})$, then $\dive_{\omega}\varphi_{\delta}=\brac{\dive_{\omega}\varphi}_{\delta}$ and $\dive_{\omega} \varphi_{\delta}\to \dive_{\omega}\varphi$. 
\item ($n=3$) If $\varphi \in \dom(\curl_{\omega})$, then $\curl_{\omega}\varphi_{\delta}=\brac{\curl_{\omega}\varphi}_{\delta}$ and $\curl_{\omega}\varphi_{\delta}\to \curl_{\omega}\varphi$.
\item ($n=3$) If $\varphi \in \dom(\curl_{\omega}^{*})$, then $\curl_{\omega}^{*}\varphi_{\delta}=\brac{\curl_{\omega}^{*}\varphi}_{\delta}$ and $\curl_{\omega}^{*}\varphi_{\delta}\to \curl_{\omega}^{*}\varphi$.
\item $\curl_\omega=\curl_\omega^*$.
\end{enumerate}
\begin{proof} The first statements of the lemma follow the same way as they follow in the deterministic case.

\textit{(a)}
We have (for $i\in \{1,\ldots,n\}$), 
\begin{align*}
\left\|\frac{1}{h}\brac{T_{h e_i}\varphi_{\delta} -\varphi_{\delta}}-\brac{\partial_{\omega,i}\varphi}_{\delta}\right\|_{L^2(\Omega)}=\left\| \int_{\R^n} \rho_{\delta}(y)\brac{ \frac{T_{y+h e_i}\varphi-T_y \varphi}{h} -T_y\partial_{\omega,i} \varphi} \right\|_{L^{2}(\Omega)}.
\end{align*}
Using the above and that $\rho_{\delta}$ is bounded and compactly supported, we obtain that there exists $C(\delta)\geq 0$ such that
\begin{equation*}
\left\|\frac{1}{h}\brac{T_{h e_i}\varphi_{\delta} -\varphi_{\delta}}-\brac{\partial_{\omega,i}\varphi}_{\delta}\right\|_{L^2(\Omega)}\leq C(\delta)\left\|\frac{1}{h}\brac{T_{h e_i}\varphi - \varphi}-\partial_{\omega,i}\varphi \right\|_{L^2(\Omega)}.
\end{equation*}
In the limit $h\to 0$ the right-hand side vanishes, and this implies that $\mathrm{grad}_{\omega}\varphi_{\delta}=(\mathrm{grad}_{\omega}\varphi)_{\delta}$. Consequently, $\mathrm{grad}_{\omega}\varphi_{\delta}\to \mathrm{grad}_{\omega}\varphi$.

\textit{(b)}
For $\eta \in H^1(\Omega)$, we have
\begin{align*}
-\expect{\varphi_{\delta},\mathrm{grad}_{\omega} \eta }_{L^2(\Omega)^n}=- \expect{\int_{\R^n}\rho_{\delta}(y)T_y\varphi \mathrm{grad}_{\omega}\eta dy} =-\int_{\R^{n}}\expect{\rho_{\delta}(y)T_y \varphi \mathrm{grad}_{\omega}\eta }dy.
\end{align*}
Using that $T_{-y}$ and (the infinitesimal generator) $\mathrm{grad}_{\omega}$ commute, we obtain that the last expression equals 
\begin{equation*}
-\int_{\R^n}\rho_{\delta}(y)\expect{\varphi \mathrm{grad}_{\omega} T_{-y}\eta}dy=\expect{\int_{\R^n}\rho_{\delta}(y)T_y \dive_{\omega} \varphi dy \eta}=\expect{\brac{\dive_{\omega}\varphi}_{\delta},\eta}_{L^2(\Omega)}.
\end{equation*}
As a result of this, we have $\dive_{\omega}\varphi_{\delta}=\brac{\dive_{\omega}\varphi}_{\delta}$ and $\dive_{\omega}\varphi_{\delta}\to \dive_{\omega}\varphi$.

\textit{(c)} Let $\eta \in \dom(\curl_{\omega}^{*})$. We have
\begin{equation*}
\expect{\varphi_{\delta}, \curl_{\omega}^{*} \eta}_{L^2(\Omega)^3}=\expect{\int_{\R^n}\rho_{\delta}(y) T_y \varphi \curl_{\omega}^{*}\eta dy}= \int_{\R^n}\rho_{\delta}(y) \expect{\varphi T_{-y}\curl_{\omega}^{*}\eta }dy.
\end{equation*}
In order to obtain the claim of the lemma, it is sufficient to show that $T_{-y}\curl_{\omega}^* \eta = \curl_{\omega}^{*} T_{-y}\eta$. Indeed, in that case the above expression equals
\begin{equation*}
\int_{\R^n} \rho_{\delta}(y) \expect{\varphi \curl_{\omega}^{*} T_{-y}\eta}dy= \expect{\int_{\R^n}\rho_{\delta}(y)T_{y}\curl_{\omega}\varphi dy \eta}= \expect{\brac{\curl_{\omega}\varphi}_{\delta},\eta}_{L^2(\Omega)^3},
\end{equation*}
and therefore the claim follows. In the following we show that $T_{-y}\curl_{\omega}^* \eta = \curl_{\omega}^{*} T_{-y}\eta$. Let $\psi \in H^1(\Omega)^3$. We have
\begin{equation*}
\expect{T_{-y} \eta, \curl_{\omega}\psi}_{L^2(\Omega)^3}= \expect{\eta, T_y \curl_{\omega} \psi}_{L^2(\Omega)^3}=\expect{\eta, \curl_{\omega} T_y \psi}_{L^2(\Omega)^3}=\expect{T_{-y}\curl_{\omega}^{*}\eta, \psi}_{L^2(\Omega)^3}
\end{equation*}
By density of $H^1(\Omega)^3$ in $\dom(\curl_{\omega})$, we obtain that for any $\psi \in \dom(\curl_{\omega})$, $\expect{T_{-y}\eta, \curl_{\omega}\psi}_{L^2(\Omega)^3}=\expect{T_{-y} \curl_{\omega}^* \eta, \psi}_{L^2(\Omega)^3}$ that implies $T_{-y}\eta \in \dom(\curl_{\omega}^*)$ and $T_{-y}\curl_{\omega}^* \eta = \curl_{\omega}^{*} T_{-y}\eta$. The proof is done.

\textit{(d)} The proof follows analogously to part \textit{(c)} if we obtain that for $\eta \in \dom(\curl_{\omega})$, $T_{-y}\curl_{\omega}\eta= \curl_{\omega} T_{-y}\eta$. In order to show this, we consider a sequence $(\eta_{k})_k$ in $H^1(\Omega)^3$ such that $\eta_{k}\to \eta$ and $\curl_{\omega}\eta_k \to \curl_{\omega}\eta$ (by definition such a sequence exists). The operator $T_{-y}$ is unitary and therefore we have $T_{-y}\eta_{k}\to T_{-y}\eta$ and $T_{-y}\curl_{\omega}\eta_k \to T_{-y}\curl_{\omega}\eta$. Moreover, since $T_{-y}\curl_{\omega}\eta_k=\curl_{\omega} T_{-y} \eta_k$, and using that $\curl_{\omega}$ is closed, we get $\curl_{\omega}T_{-y}\eta = T_{-y}\curl_{\omega}\eta$.

\textit{(e)} It is elementary to show that $\curl_\omega\subseteq \curl_\omega^*$; see also \cite[p 22]{Bourgeat1994}. The remaining inclusion follows from \textit{(d)} as this statement shows that $H^1(\Omega)$ is a core for $\curl_\omega^*$. Hence,
\[
    \curl_\omega\subseteq \curl_\omega^* = \overline{\curl_\omega^*|_{H^1(\Omega)}}=\overline{\curl_\omega|_{H^1(\Omega)}}=\curl_\omega.\qedhere
\]
\end{proof}
\end{lemma}

\begin{proof}[Proof of Lemma \ref{lemma:981}]
Let $\varphi \in \kar(\curl_{\omega})$. Then $\varphi= P_{\mathsf{inv}} \varphi + \Psi$ where $\Psi= \varphi- P_{\mathsf{inv}}\varphi$. Using (\ref{equation:957}) it follows that $\Psi$ may be decomposed as follows
\begin{equation*}
\Psi = \Psi_1 + \Psi_2, 
\end{equation*}
where $\Psi_{1}\in \kar(\dive_{\omega})$ and $\Psi_2 \in \overline{\mathsf{ran}}(\mathrm{grad}_{\omega})$. In the following we show that $\Psi_1 = 0$. Using Lemma \ref{lemma:996} we find a sequence $(\varphi_{k})_k$ in $H^{\infty}(\Omega)$ such that $\mathrm{grad}_{\omega}\varphi_{k}\to \Psi_2$ as $k \to \infty$. Moreover, since $\curl_{\omega}\mathrm{grad}_{\omega}\varphi_{k}=0$ for all $k\in \mathbb{N}$, we conclude that $\Psi_2 \in \kar(\curl_{\omega})$ and therefore it follows that \[\Psi_1 = \Psi-\Psi_2 = \varphi - P_{\mathsf{inv}}\varphi -\Psi_2 \in \kar(\curl_{\omega}).\] For the mollified functions $\Psi_{1,\delta}$ (defined as in (\ref{eq:992})), by a direct computation we obtain
\begin{equation*}
\int_{\Omega}|\curl_{\omega}\Psi_{1,\delta}|^2+ |\dive_{\omega}\Psi_{1,\delta}|^2 = \int_{\Omega}|\mathrm{grad}_{\omega} \Psi_{1,\delta}|^2.
\end{equation*}
Using Lemma \ref{lemma:996} we pass to the limit $\delta \to 0$ and it follows that $\mathrm{grad}_{\omega}\Psi_{1,\delta} \to 0$. Moreover, since $\Psi_{1,\delta}\to \Psi_{1}$, we obtain that $\mathrm{grad}_{\omega}\Psi_{1}= 0$ and therefore $\Psi_{1}=0$ (using that $P_{\mathsf{inv}}\Psi_1=0$). This concludes the proof of the first part. The second claim follows from Lemma \ref{lemma:996} \textit{(e)}.
\end{proof}

\subsection{Periodic homogenization of elliptic equations}\label{sec:elliptic}
We start by showing that the classical example of periodic homogenization of elliptic equations fits into the previously described abstract framework (see Section \ref{sec:4}). We refer to \cite{Bensoussan1978,Allaire1992} for the standard treatment of elliptic equations with periodic coefficients and to \cite{Papanicolaou1979,Bourgeat1994} for its stochastic counterpart. Let $Q\subset \R^n$ be open and bounded. In this section, we consider the periodic unfolding operator $\unf: L^2(Q)\otimes L^2(\Box)\to L^2(Q)\otimes L^2(\Box)$ defined in (\ref{equation:869}).  

In this setting the role of $C_d$ and $C_s$ is played by $\nablao$ and $\mathrm{grad}_y^{\#}$, respectively.
Let $A \in L^\infty(Q\times\Box)^{n\times n}$ be such that there exists $c>0$ with $|A(x,y)| \leq \frac{1}{c}$ and $\frac{1}{2}\brac{A(x,y)+A(x,y)^*}\geq c$ a.e. We interpret $A$ as a multiplication operator in $(L^2(Q) \otimes L^2(\Box))^n$. For $\varepsilon>0$, we consider the following equation
\begin{equation}\label{eq:1154}
-\dive_x \mathcal{T}_{-\varepsilon} A \mathcal{T}_{\varepsilon}\nablao u_\varepsilon = f
\end{equation}
with $f \in L^2(Q)\otimes \C$. In this case, the term $\mathcal{T}_{-\varepsilon}A \unf \nablao u_\varepsilon$ boils down to the familiar expression $(x,y)\mapsto A\brac{x, y+ \cb{\frac{x}{\varepsilon}}_{\Box}}\nablao u_\varepsilon(x,y)$. Since $\nablao$ is injective and its range is closed (due to the Poincar{\'e} inequality), a direct application of Theorem \ref{thm:ellhom} and the fact that $\unf u = u$ for $u \in \kar(\mathrm{grad}_y^{\#})$ imply the following:
\begin{corollary}\label{cor:ellhom}
Let $A$ be given as above and $u_{\varepsilon}$ be the unique solution of (\ref{eq:1154}). Then
\begin{equation*}
\mathcal{T}_\varepsilon u_{\varepsilon}\to u \text{ strongly in }L^2(Q)\otimes L^2(\Box), \quad  \unf \nablao u_{\varepsilon} \rightharpoonup \nablao u + v \quad \text{weakly in }(L^2(Q)\otimes L^2(\Box))^n,
\end{equation*}  
where $u \in \dom(\nablao)\otimes \C$ and $v \in \overline{\rge}(\mathrm{grad}_y^{\#})$ are the unique solution to
\begin{align}\label{equation:889}
\begin{split}
-\dive_x P A (\nablao u + v)= f,\\
-\dive_y^{\#} A (\nablao u + v) = 0.
\end{split}
\end{align}
\end{corollary}
Above, $P$ is the projection to constant functions (in the $y$-variable), i.e., $P\varphi = \int_{\Box}\varphi dy$.

\begin{remark} In order to transform the above (two-scale) homogenized problem (\ref{equation:889}) into the usual one-scale form (see \cite{Allaire1992} for detailed investigation of such two-scale effective equations), we might introduce the following (uniquely defined) correctors  $\varphi_i \in \dom(\mathrm{grad}_y^{\#})$ ($i\in \cb{1,\ldots,n}$) by
\begin{equation*}
- \dive^{\#}_{y} A \brac{e_i + \mathrm{grad}_y^{\#} \varphi_i }=0, \quad \int_{\Box} \varphi_i = 0.
\end{equation*} 
It follows that the choice $v = \sum_{i=1}^{n} (\nablao u)_i \mathrm{grad}_y^{\#} \varphi_i$ satisfies the second equation in (\ref{equation:889}). Consequently, $u$ is the solution of the equation
\begin{equation*}
-\dive_{x} A_{\mathsf{hom}}\nablao u = f,
\end{equation*}
where $A_{\mathsf{hom}}e_i \cdot e_j = \int_{\Box}A(e_i + \mathrm{grad}_y^{\#} \varphi_i)\cdot (e_j + \mathrm{grad}_y^{\#} \varphi_j)dy$. The last equation is the classical form of the homogenized equation. Notice, however, that the result presented in Corollary \ref{cor:ellhom} is slightly different from the classical result; as the notion of convergence is different. We also refer to \cite[Remark on page 32]{Bourgeat1994} for a similar statement in the stochastic setting.
\end{remark}

\subsection{Stochastic homogenization of Maxwell's equations}\label{sec:maxwell}
In this section we consider stochastic homogenization of Maxwell's equations. We refer to \cite{Wellander2001,kristensson2003homogenization} for the treatment of Maxwell's equations in the periodic setting using two-scale convergence arguments. The stochastic-periodic case is treated in \cite{tachago2017stochastic} that is based on the notion of stochastic two-scale convergence from \cite{Bourgeat1994}. However, our approach is different, it relies on the operator theoretic formulation and on the unfolding strategy, in fact, the homogenization result Corollary \ref{cor:1199} readily follows from the abstract Theorem \ref{thm:mr}. Also, we treat a more general situation---in contrast to our case---in \cite{tachago2017stochastic} the assumptions on the coefficients do not allow for jumps or for regions, where the conductivity or dielectricity vanish. Moreover, as far as we know, the corrector type result Corollary \ref{cor:1326} is not presented earlier in the stochastic setting, see \cite[Theorem 3.3]{Wellander2001} for a similar periodic corrector result.

Let $Q\subseteq \R^3$ be open and $(\Omega,\Sigma, \mu, \tau)$ be a probability space satisfying Assumption \ref{assumpt:267} (with $n=3$). In this section we consider the stochastic unfolding operator $\unf: L^2(Q)\otimes L^2(\Omega)\to L^2(Q)\otimes L^2(\Omega)$ defined in (\ref{eq:298}). The role of $C_d$ and $C_s$ in this setting is played by the operators $\curlo$ and $\curl_{\omega}$, respectively. 

We set $H_0 = H_1 = (L^2(Q)\otimes L^2(\Omega))^3$ and $H= H_0\oplus H_1$. We consider $\nu_0> 0$, $\eta_0,\sigma_0,\mu_0 \in  L^{\infty}(Q\times \Omega)^{3\times 3}$ such that $\eta_0$ and $\mu_0$ are Hermitian a.e. and there exists $c>0$ such that
\begin{equation}\label{eq:etasigmamu}
\nu \eta_0 + \Re(\sigma_0) \geq c \quad (\text{for all }\nu > \nu_0), \quad \mu_0 \geq c.
\end{equation}
For $\varepsilon>0$ and $(f,g)\in L^2_{\nu}\brac{\R; L^2(Q)^3 \oplus L^2(Q)^3}$, we consider the following system of equations
\begin{align}\label{equation:935}
\overline{B}_{\varepsilon} \begin{pmatrix} u_{\varepsilon} \\ q_{\varepsilon}  \end{pmatrix}= \begin{pmatrix} f \\ g \end{pmatrix}, \; \text{where } B_{\varepsilon}= \mathcal{T}_{-\varepsilon} \begin{pmatrix}
      \partial_{t,\nu} \eta_0 + \sigma_0 & 0 \\ 0 & \partial_{t,\nu} \mu_0
   \end{pmatrix}  \unf + \begin{pmatrix} 0 & -\curl_x\\ \curlo & 0 \end{pmatrix}.
\end{align}
The above system represents a system of Maxwell's equations with random and oscillating coefficients (the oscillating coefficients have the form $\mathcal{T}_{-\varepsilon} \eta_0 \unf$ and similarly for $\sigma_0$ and $\mu_0$). According to Theorem \ref{thm:st} there exists a unique solution of the above equation $(u_{\varepsilon},q_{\varepsilon}) \in L^2_{\nu}(\R; H)$. Moreover, a direct application of Theorem \ref{thm:mr} yields the following:
\begin{corollary}\label{cor:1199}
Let $\eta_0, \sigma_0,\mu_0$ be given as above and let $(u_{\varepsilon},q_{\varepsilon})\in L^2_{\nu}(\R; H)$ be the unique solution to (\ref{equation:935}). We have
\begin{equation*}
\mathcal{T}_{\varepsilon}(u_{\varepsilon},q_{\varepsilon}) \rightharpoonup (\iota_s u, \iota_{s} q) \quad \text{weakly in }L^2_{\nu}(\R, H).
\end{equation*}
Above, $\iota_s$ denotes the canonical embedding $\iota_s: \kar(\curl_{\omega}) \hookrightarrow H_0=H_1$ and \[(u,q) \in L^2_{\nu}\brac{\R;\kar(\curl_{\omega}) \oplus \kar(\curl_{\omega})}\] denotes the unique solution to 
\begin{equation}\label{equation:966}
\overline{B}_{\mathsf{hom}}\begin{pmatrix} u \\ q  \end{pmatrix}= \begin{pmatrix} \iota_s^{*} f \\ \iota_{s}^* g \end{pmatrix}, \; \text{where }B_{\mathsf{hom}}= \begin{pmatrix}
      \partial_{t,\nu} \iota_s^{*} \eta_0 \iota_s + \iota_s^{*} \sigma_0 \iota_s & 0 \\ 0 & \partial_{t,\nu} \iota_{s}^*\mu_0 \iota_{s}
   \end{pmatrix} + \begin{pmatrix} 0 & - \overline{\iota_s^{*}\curl_x \iota_{s}}\\ \overline{\iota_{s}^* \curlo \iota_s} & 0 \end{pmatrix}.
\end{equation}
\end{corollary}

Note that the fact $(u,q) \in L_\nu^2(\R; \kar(\curl_{\omega}) \oplus \kar(\curl_{\omega}))$ does \emph{not} imply in general that the limit solution is deterministic (or shift-invariant), i.e., the solution still depends on the probability space variable $\omega$. In the following we present an equivalent formulation of the above equation by decomposing the solution $(u,q)$ to its deterministic (shift-invariant) and random parts. In particular, the deterministic part satisfies an effective Maxwell system and the random part is given by a suitable corrector equation.
We introduce the following canonical embeddings
\begin{align*}
\iota_{0}: L^2_{\mathsf{inv}}(\Omega)^3 \hookrightarrow \kar(\curl_{\omega}), \quad 
\iota_{1}: L^2_{\mathsf{pot}}(\Omega) \hookrightarrow \kar(\curl_{\omega}).
\end{align*}
Moreover, we consider the transformation 
\[T\coloneqq \begin{pmatrix} \iota_{0} \; & \iota_{1} \; & 0 \; & 0 \\ 0 \; &0 \; &\iota_{0} \; &\iota_{1} \end{pmatrix} : L^2_{\mathsf{inv}}(\Omega)^3 \oplus L^2_{\mathsf{pot}}(\Omega)\oplus L^2_{\mathsf{inv}}(\Omega)^3 \oplus L^2_{\mathsf{pot}}(\Omega)\to \kar(\curl_{\omega})\oplus \kar(\curl_{\omega})\] that is unitary (by Lemma \ref{lemma:981}). As a result of this, we obtain the following:
\begin{corollary}
Let $(u,q) \in L^2_{\nu}\brac{\R;\kar(\curl_{\omega}) \oplus \kar(\curl_{\omega})}$ be the solution of (\ref{equation:966}). Then
 \[
 \begin{pmatrix}u_0 \\ \chi_1 \\q_0 \\ \chi_2\end{pmatrix}\coloneqq T^{-1}(u,q) \in L^2_{\nu}\left(\R; \big( (L^2(Q)\otimes L^2_{\mathsf{inv}}(\Omega))^3 \oplus (L^2(Q)\otimes L^2_{\mathsf{pot}}(\Omega))\big)^2\right)
 \] is the unique solution to  
\begin{equation}\label{equation:997}
\overline{T^{-1}B_{\mathsf{hom}} T}\begin{pmatrix} u_0 \\ \chi_1 \\ q_0 \\ \chi_2 \end{pmatrix}= T^{-1}\begin{pmatrix} \iota_s^* f \\ \iota_{s}^* g \end{pmatrix}.
\end{equation}
\end{corollary}
\begin{remark}\label{rem:1214}
Dropping the notation for the embeddings $\iota_s, T$ and closure bars, system (\ref{equation:997}) reads 
\begin{align*}
\partial_{t,\nu}  P_{\mathsf{inv}} \eta_0 (u_0+ \chi_1) + P_{\mathsf{inv}}\sigma_0 \brac{u_0+\chi_1}-\curl_{x} q_0 & = f\\
\partial_{t,\nu} P_{\mathsf{pot}}\eta_0 (u_0+ \chi_1)+ P_{\mathsf{pot}}\sigma_0 (u_0+ \chi_1) & =0 \\
\partial_{t,\nu} P_{\mathsf{inv}}\mu_0 (q_0+\chi_2)+ \curlo u_0 &= g \\
\partial_{t,\nu} P_{\mathsf{pot}}\mu_0 (q_0+ \chi_2) & = 0.  
\end{align*}
Note that in the second equation we used that $P_{\mathsf{pot}}\curl_{x}(q_0+\chi_2)=0$ since $P_{\mathsf{pot}}q_0= 0$ and $P_{\mathsf{pot}}\curl_x\chi_2=0$ (that can be obtained by a direct computation) and in the fourth equation we use $P_{\mathsf{pot}}\curlo (u_0+\chi_1)=0$ (obtained similarly as the previous claim).
We regard the first and third equations as the effective Maxwell system for the (averaged) variable $(u_0,q_0)$ and the second and fourth equations are corrector equations which allow us to express $(\chi_1,\chi_2)$ as functions of $(u_0,q_0)$. In particular, the second and fourth equations imply that (additionally assuming that $\eta_0$ is positive-definite)
\begin{align*}
\chi_1(t) &=-\chi_{\eta_0}^i u_0^i(t) + e^{-A t}\brac{\chi_{\eta_0}^i u_0^{i}(0)+\chi_1(0)}+\int_{0}^t e^{-A(t-\tau)}A\brac{\chi_{\eta_0}^i-\chi_{\sigma_0}^i}u_0^i(\tau)d\tau, \\
\chi_2(t) &=
-\chi_{\mu_0}^{i}q_0^i(t)+\chi_{\mu_0}^i q_0^i(0)+\chi_2(0),
\end{align*}
where $A= \eta_0^{-1} \sigma_0$ and $\chi_{\eta_0}^i,\chi_{\sigma_0}^i,\chi_{\mu_0}^i\in L^2(Q)\otimes L^2_{\mathsf{pot}}(\Omega)$ ($i\in\{1,2,3\}$) are the unique solutions to (respectively) 
\begin{equation}\label{eq:1230}
P_{\mathsf{pot}}\eta_0 (e_i-\chi_{\eta_0}^i)=0, \quad P_{\mathsf{pot}}\sigma_0 (e_i-\chi_{\sigma_0}^i)=0, \quad P_{\mathsf{pot}}\mu_0 (e_i-\chi_{\mu_0}^i)=0.
\end{equation}
The derivation of these formulas is analogous to the periodic case \cite{Wellander2001}
 (see also \cite{tachago2017stochastic} for the periodic-stochastic case).
In this sense, the effective equations for the averaged variables $(u_0,q_0)$ are non-local in time. This so-called memory effect has also been observed in \cite{W16_HPDE} and \cite{W18_NHC} for coefficients that are not necessarily periodic or stochastic.
\end{remark}
\begin{remark}[Corrector equations]
The equations in \eqref{eq:1230} are standard corrector equations in stochastic homogenization and might be brought into the form of an elliptic partial differential equation on $\R^3$. For simplicity let us consider the first equation in \eqref{eq:1230} and assume that $\eta_0$ does not depend on the physical space variable, i.e., $\eta_0(x,\omega)=\eta_0(\omega)$. The equation is equivalent to
\begin{equation}\label{eq:1232}
\int_{\Omega} \eta_0(\omega) (e_i - \chi_{\eta_0}^i(\omega))\cdot \chi(\omega)d\mu(\omega) = 0 \quad \text{for all }\chi \in L^2_{\mathrm{pot}}(\Omega),
\end{equation}
which is a variational problem in the $L^2$-probability space. It turns out that for $\mu$-almost every $\omega\in\Omega$, the vector field $x\mapsto \chi_{\eta_0}^i(\tau_x\omega)$ has a potential $\varphi(\omega,\cdot)\in H^1_{\mathrm{loc}}(\R^3)$ that is a distributional solution to
\begin{equation}\label{eq:1236}
- \nabla\cdot \brac{\eta_0(\tau_x\omega) \brac{e_i - \nabla \varphi(\omega,x)}}=0 \quad \text{in }\R^3.
\end{equation}
On the other hand, by standard theory in stochastic homogenization there exists a unique random field $\varphi:\Omega\times\R^3\to\R$ that is a  distributional solution to \eqref{eq:1236} for $\mu$-a.e.~$\omega\in\Omega$ such that $\nabla\varphi$ is stationary (i.e., $\nabla\varphi(\omega,x+z)=\nabla\varphi(\tau_z\omega,x)$ for $\mu$-a.e.~$\omega\in\Omega$ and a.e.~$x,z\in\R^3$), square integrable $\int_\Omega|\nabla\varphi(\omega,x)|^2\,d\mu(\omega)<\infty$, mean-free $\int_\Omega\nabla\varphi\,d\mu=0$, and anchored in the sense of $\int_{(0,1)^d}\varphi(\omega,x)\,dx=0$, $\mu$-a.s. With this solution we recover the random vector field $\chi_{\eta_0}^i(\omega):=\nabla\varphi(\omega,0)$ solving \eqref{eq:1232}, see \cite[Section 2.2]{neukamm2018introduction} for details. 

Note that both formulations, \eqref{eq:1232} and \eqref{eq:1236}, are not accessible to a direct numerical approximation, since in the case of \eqref{eq:1232} a typical example for the probability measure $\mu$ would be a product-measure of the form $\hat\mu^{\otimes \R^3}$, while \eqref{eq:1236} is posed on the unbounded domain $\R^3$. A standard approximation scheme is the so-called periodization method, where in \eqref{eq:1236} the domain $\R^3$ is replaced by a large torus, say $L\Box$ with $L\gg 1$, see, e.g., \cite{bourgeat2004approximations}.
\end{remark}

\textbf{Corrector type results for Maxwell equations.} First, we recall a (standard) mollification procedure from \cite{picard2017maximal} (see also \cite{waurick2015non,PTWW13_NA}). For $\delta>0$, we consider the bounded operator $\mathcal{F}_{\delta}:= \brac{1+ \delta \partial_{t,\nu}}^{-1}$. We remark that $\mathcal{F}_{\delta} \to 1$ as $\delta\to 0$ in the strong operator topology and that for any $\varphi \in \dom(\overline{B}_{\varepsilon})$, we have $\mathcal{F}_{\delta}\varphi \in \dom(B_{\varepsilon})$ and $\mathcal{F}_{\delta}\overline{B}_{\varepsilon}\varphi = B_{\varepsilon} \mathcal{F}_{\delta}\varphi$  and the same statements hold if we replace $B_{\varepsilon}$ by $B_{\mathsf{hom}}$. Using $\mathcal{F}_{\delta}$ and the properties of $(u_{\varepsilon},q_{\varepsilon})$, we obtain the following corrector type results. Before that, we provide another auxiliary result.
\begin{lemma}\label{lem:curlclo} We have
\[
   \overline{\iota_s^{*}\curl_x \iota_{s}} = \iota_s^*\curl_x \iota_s\text{ and }
      \overline{\iota_s^{*}\curlo \iota_{s}} = \iota_s^*\curlo \iota_s.
\]
\end{lemma}
\begin{proof}
  Since $\curl_x$ and $\curlo$ are closed linear operators, the claim follows, if we show that the application of $\curl_x$ and $\curlo$ leaves $\kar(\curl_\omega)$ invariant. This, however, follows from the fact that $\curl_x$ and $\curl_\omega$ ($\curlo$ and $\curl_\omega$) commute on the intersection of their domains.
\end{proof}
In order to avoid clutter in notation, in the following we disregard the notation for the embedding $\iota_s$ and $T$.
\begin{proposition}\label{prop:1277}
Let $\delta>0$, $(u_{\varepsilon}, q_{\varepsilon})$ be the solution of (\ref{equation:935}) and $(u_0,\chi_1, q_0, \chi_2)$ be the solution of (\ref{equation:997}). Then
\begin{equation*}
\|\mathcal{F}_{\delta}\brac{u_{\varepsilon} - \mathcal{T}_{-\varepsilon}\chi_1 - u_0} \|_{L^2_{\nu}(\R; H_0)} + \|\mathcal{F}_{\delta}\brac{q_{\varepsilon}-\mathcal{T}_{-\varepsilon}\chi_2 - q_0}\|_{L^2_{\nu}(\R; H_1)} \to 0 \quad \text{as }\varepsilon\to 0.
\end{equation*}
\end{proposition}
\begin{proof}
Using the facts that 
\[
\Re \left\langle \begin{pmatrix}u_\epsilon \\ q_\epsilon \end{pmatrix},\invunf \begin{pmatrix}   \partial_{t,\nu} \eta_0 + \sigma_0 \; & 0 \\  0 \; & \partial_{t,\nu} \mu_0 \end{pmatrix}\unf \begin{pmatrix}u_\epsilon \\ q_\epsilon \end{pmatrix}\right\rangle \geq c \left\|\begin{pmatrix}u_\epsilon \\ q_\epsilon \end{pmatrix}\right\|^2
\]
 (using the assumptions on $\eta,\sigma, \mu$, see \eqref{eq:etasigmamu}) and that $\begin{pmatrix}0 & -\curl_{x}\\
\curlo & 0 \end{pmatrix}$ is skew-self-adjoint, we obtain
\begin{eqnarray}\label{equation:1041}
&& \|\mathcal{F}_{\delta}\brac{ u_{\varepsilon} - \mathcal{T}_{-\varepsilon}\chi_1 - u_0 }\|^2_{L^2_{\nu}(\R; H_0)} + \|\mathcal{F}_{\delta}\brac{q_{\varepsilon}-\mathcal{T}_{-\varepsilon}\chi_2 - q_0}\|^2_{L^2_{\nu}(\R; H_1)} \\
& \leq & \frac{1}{c} \Re \expect{B_{\varepsilon} \mathcal{F}_{\delta}\begin{pmatrix}u_{\varepsilon} - \mathcal{T}_{-\varepsilon}\chi_1-u_0 \\ q_{\varepsilon}-\mathcal{T}_{-\varepsilon}\chi_2-q_0 \end{pmatrix},\mathcal{F}_{\delta}\begin{pmatrix}u_{\varepsilon} - \mathcal{T}_{-\varepsilon}\chi_1-u_0 \\ q_{\varepsilon}-\mathcal{T}_{-\varepsilon}\chi_2-q_0\end{pmatrix}} \nonumber\\
& = & \frac{1}{c} \Re \expect{\mathcal{F}_{\delta} \overline{B}_{\varepsilon}\begin{pmatrix}  u_{\varepsilon} \\ q_{\varepsilon}\end{pmatrix},\mathcal{F}_{\delta}\begin{pmatrix} u_{\varepsilon}-\mathcal{T}_{-\varepsilon}\chi_1- u_0\\ q_{\varepsilon}- \mathcal{T}_{-\varepsilon}\chi_2 - q_0 \end{pmatrix}} \nonumber \\
& & - \frac{1}{c} \Re \expect{B_{\varepsilon}\mathcal{F}_{\delta} \begin{pmatrix}  u_0+ \mathcal{T}_{-\varepsilon}\chi_1 \\ q_0 + \mathcal{T}_{-\varepsilon}\chi_2 \end{pmatrix},\mathcal{F}_{\delta}\begin{pmatrix} u_{\varepsilon}-\mathcal{T}_{-\varepsilon}\chi_1- u_0\\ q_{\varepsilon}- \mathcal{T}_{-\varepsilon}\chi_2 - q_0 \end{pmatrix}}
. \nonumber
\end{eqnarray}
Above, we use that $\mathcal{F}_{\delta}\begin{pmatrix} \invunf \chi_1 + u_0 \\ \invunf \chi_2 + q_0 \end{pmatrix} \in \dom(B_{\varepsilon})$.
This can be seen using the facts that \[\mathcal{F}_{\delta}\begin{pmatrix} u_0+ \chi_1\\ q_0+ \chi_2 \end{pmatrix} \in \dom(B_{\mathsf{hom}})\] and that $\invunf$ and $\mathcal{F}_{\delta}$ commute and, thus, 
\[
\mathcal{F}_{\delta}\begin{pmatrix} \invunf \chi_1 + u_0 \\ \invunf \chi_2 + q_0 \end{pmatrix} \in \dom\left(\begin{pmatrix}0 & -\curl_{x}\\
\curlo & 0 \end{pmatrix}\right),
\]
by Lemma \ref{lem:curlclo}.
Since $(u_{\varepsilon},q_{\varepsilon})$ solves problem (\ref{equation:935}) and $(f,g)\in \kar(\mathcal{T}_{\varepsilon}-1)$, it follows that the first term on the right-hand side satisfies
\begin{align*}
\frac{1}{c} \Re \expect{\mathcal{F}_{\delta} \overline{B}_{\varepsilon}\begin{pmatrix}  u_{\varepsilon} \\ q_{\varepsilon}\end{pmatrix},\mathcal{F}_{\delta}\begin{pmatrix} u_{\varepsilon}-\mathcal{T}_{-\varepsilon}\chi_1- u_0\\ q_{\varepsilon}- \mathcal{T}_{-\varepsilon}\chi_2 - q_0 \end{pmatrix}} & = \frac{1}{c} \Re \expect{\mathcal{F}_{\delta} \begin{pmatrix}  f \\ g \end{pmatrix},\mathcal{F}_{\delta}\begin{pmatrix} u_{\varepsilon}-\mathcal{T}_{-\varepsilon}\chi_1- u_0 \\ q_{\varepsilon}- \mathcal{T}_{-\varepsilon}\chi_2 - q_0 \end{pmatrix}}\\
& = \frac{1}{c} \Re \expect{\mathcal{F}_{\delta} \begin{pmatrix}  f \\ g \end{pmatrix},\mathcal{F}_{\delta}\begin{pmatrix} \unf u_{\varepsilon}-\chi_1- u_0 \\ \unf q_{\varepsilon}- \chi_2 - q_0 \end{pmatrix}}.
\end{align*}
The last expression vanishes in the limit $\varepsilon \to 0$ since $(u_{\varepsilon},q_{\varepsilon})\overset{2}{\rightharpoonup}(u_0+\chi_1, q_0+\chi_2)$. We treat the second term on the right-hand side of (\ref{equation:1041}) as follows:
\begin{alignat}{3}
\begin{aligned}\label{eq:1302}
& \Re \expect{B_{\varepsilon} \mathcal{F}_{\delta} \begin{pmatrix}  u_0+ \mathcal{T}_{-\varepsilon}\chi_1 \\ q_0 + \mathcal{T}_{-\varepsilon}\chi_2 \end{pmatrix},\mathcal{F}_{\delta}\begin{pmatrix} u_{\varepsilon}-\mathcal{T}_{-\varepsilon}\chi_1- u_0\\ q_{\varepsilon}- \mathcal{T}_{-\varepsilon}\chi_2 - q_0 \end{pmatrix}}\\ & =  \Re \expect{\brac{\begin{pmatrix} \partial_{t,\nu}\eta + \sigma & 0 \\ 0 & \partial_{t,\nu}\mu  \end{pmatrix} + \unf \begin{pmatrix}0 & -\curl_{x}\\ \curlo & 0  \end{pmatrix}\mathcal{T}_{-\varepsilon}} \mathcal{F}_{\delta} \begin{pmatrix} u_0 + \chi_1\\ q_0+ \chi_2 \end{pmatrix},\mathcal{F}_{\delta}\begin{pmatrix} \unf u_{\varepsilon}-\chi_1- u_0 \\ \unf q_{\varepsilon}- \chi_2 - q_0 \end{pmatrix}}  \\
& =  \Re \expect{\brac{\begin{pmatrix} \partial_{t,\nu}\eta + \sigma & 0 \\ 0 & \partial_{t,\nu}\mu  \end{pmatrix} +  \begin{pmatrix}0 & -\curl_{x}\\ \curlo & 0  \end{pmatrix}} \mathcal{F}_{\delta} \begin{pmatrix} u_0 + \chi_1\\ q_0+ \chi_2 \end{pmatrix},\mathcal{F}_{\delta}\begin{pmatrix} \unf u_{\varepsilon}-\chi_1- u_0 \\ \unf q_{\varepsilon}- \chi_2 - q_0 \end{pmatrix}}. 
\end{aligned}
\end{alignat}
In order to justify the second equality above, we compute
\begin{equation}\label{eq:1275}
\begin{pmatrix} 0 \; & -\curl_{x}\\ \curlo \; & 0  \end{pmatrix}\invunf \mathcal{F}_{\delta} \begin{pmatrix} u_0+\chi_1 \\ q_0+ \chi_2 \end{pmatrix}= \begin{pmatrix} \curl_{x} \invunf \mathcal{F}_{\delta} \brac{q_0 + \chi_2}\\ \curlo \invunf \mathcal{F}_{\delta}\brac{u_0+ \chi_1} \end{pmatrix}= \begin{pmatrix} \invunf \curl_{x} \mathcal{F}_{\delta} \brac{q_0 + \chi_2}\\ \invunf \curlo \mathcal{F}_{\delta}\brac{u_0+ \chi_1} \end{pmatrix},
\end{equation} 
where we used the commutation relations from Lemma \ref{lemma:964} (b), the fact that $\mathcal{F}_{\delta}$ and $\curl_{\omega}$ commute, and that $(u_0+\chi_1)\in \kar(\curl_{\omega})$ and $\brac{q_0 + \chi_2} \in \kar(\curl_{\omega})$.
Also, letting $\varepsilon\to 0$ in (\ref{eq:1302}), the right-hand side vanishes since $(u_{\varepsilon}, q_{\varepsilon})\overset{2}{\rightharpoonup}(u_0+ \chi_1,q_0+ \chi_2)$. This concludes the proof.
\end{proof}
If we, additionally, assume that $(f,g)\in \dom(\partial_{t,\nu})$, it follows that $(u_{\varepsilon},q_{\varepsilon})\in \dom(\partial_{t,\nu})$ and $(u_{\varepsilon}, q_{\varepsilon})\in \dom(B_{\varepsilon})$ (the analogous claims hold if we replace $B_{\varepsilon}$ by $B_{\mathsf{hom}}$ and $(u_{\varepsilon},q_{\varepsilon})$ by $(u_0+\chi_1,q_0+\chi_2)$).
As a result of this, we obtain the following:
\begin{corollary}\label{cor:1326}
Assume the same assumptions as in Corollary \ref{cor:1199}. Additionally, we assume that $(f,g) \in H^1_{\nu}(\R;H) (= \dom(\partial_{t,\nu}))$. Then 
\begin{equation*}
\| u_{\varepsilon} - \mathcal{T}_{-\varepsilon}\chi_1 - u_0 \|^2_{L^2_{\nu}(\R; H_0)} + \| q_{\varepsilon}-\mathcal{T}_{-\varepsilon}\chi_2 - q_0 \|^2_{L^2_{\nu}(\R; H_1)} \to 0 \quad \text{as }\varepsilon\to 0.
\end{equation*}
\begin{proof}
Since $(f,g)\in H^1_{\nu}(\R; H)$, we may apply the operator $(1+\partial_{t,\nu})$ to both sides of the equations (\ref{equation:935}) and (\ref{equation:966}) to obtain

\begin{align*}
\overline{B}_{\varepsilon}(1+ \partial_{t,\nu})\begin{pmatrix} u_{\varepsilon}\\ q_{\varepsilon}\end{pmatrix}=(1+\partial_{t,\nu})\begin{pmatrix}f \\ g \end{pmatrix}, \\
\overline{B}_{\mathsf{hom}} (1+\partial_{t,\nu})\begin{pmatrix}
u_0 + \chi_1\\ q_0 + \chi_2
\end{pmatrix} = (1+\partial_{t,\nu})\begin{pmatrix} f \\ g  \end{pmatrix}.
\end{align*}
We have that $w_{\varepsilon}:= \brac{1+ \partial_{t,\nu}}\begin{pmatrix} u_{\varepsilon} \\ q_{\varepsilon} \end{pmatrix}$ and $w_0 :=\brac{1+ \partial_{t,\nu}} \begin{pmatrix} u_0+\chi_1 \\ q_0+ \chi_2 \end{pmatrix}$ solve equations (\ref{equation:935}) and (\ref{equation:966}) with right-hand side $\brac{1+\partial_{t,\nu}}\begin{pmatrix}f \\ g \end{pmatrix}$ (instead of $\begin{pmatrix}f \\ g \end{pmatrix}$). As a result of this, 
Proposition \ref{prop:1277} implies that for any $\delta>0$, we have
\begin{equation*}
\|\mathcal{F}_{\delta}\brac{w_{\varepsilon} - \mathcal{T}_{-\varepsilon}w_0} \|^2_{L^2_{\nu}(\R; H)} \to 0 \quad \text{as }\varepsilon\to 0.
\end{equation*}
Setting $\delta=1$, the claim follows.
\end{proof}
\end{corollary}

\noindent
\subsection{Stochastic homogenization of some mixed type equations}\label{sec:wave}
A standard reference for homogenization of the wave and heat equation is \cite{brahim1992correctors}, where general equations are treated in the framework of H-convergence. In contrast to standard results for the wave or heat equation, the equation we consider features coefficients $\eta$ and $\sigma$, cf. \eqref{eq:1325}, that are allowed to be alternatingly vanishing in some regions of the considered physical domain. This means that our setting contains equations of mixed alternating hyperbolic-parabolic type.

Let $Q \subseteq \R^n$ be open and $(\Omega,\Sigma, \mu, \tau)$ be a probability space satisfying Assumption \ref{assumpt:267}. We consider the stochastic unfolding operator $\unf: L^2(Q)\otimes L^2(\Omega)\to L^2(Q)\otimes L^2(\Omega)$ defined in \eqref{eq:298}. The role of $C_d$ and $C_s$ in this setting is played by the operators $\nablao$ and $\mathrm{grad}_{\omega}$ (respectively).

Let $\nu_0>0$, $A \in L^\infty(Q\times\Omega)^{n\times n}$ be such that there exists $c>0$ with $A$ Hermitian a.e and $c \leq A(x,\omega)\leq \frac{1}{c}$ a.e.. Also, let $\eta,\sigma \in L^{\infty}(Q \times \Omega)$ be such that there exists $c>0$ with
\begin{equation*}
\Re(z \eta + \sigma) \geq c \quad \text{(for all $z \in  \C_{\Re\geq\nu_0}$)}.
\end{equation*}

  Let $H_0=L^2(Q)\otimes L^2(\Omega)$, $H_1= L^2(Q)\otimes L^2(\Omega)^{n}$ and $H=H_0\oplus H_1$. For $\varepsilon>0$ and $f\in L^2(Q)$, we consider the following system
\begin{equation}\label{eq:1345}
\overline{B}_{\varepsilon} \begin{pmatrix}u_{\varepsilon} \\ q_{\varepsilon} \end{pmatrix}=\begin{pmatrix}
f \\ 0
\end{pmatrix}, \quad \text{where }B_{\varepsilon}= \mathcal{T}_{-\varepsilon} \begin{pmatrix} \partial_{t,\nu} \eta + \sigma  \; & 0 \\ 0 \; & \partial_{t,\nu} A^{-1} \end{pmatrix} \unf + \begin{pmatrix} 0 \; & \dive_{x} \\ \nablao \; & 0 \end{pmatrix}.
\end{equation}
According to Theorem \ref{thm:st} the above system has a unique solution $(u_{\varepsilon},q_{\varepsilon}) \in L^2_{\nu}(\R; H)$. 
\begin{remark}
Setting $w_{\varepsilon}=\partial_{t,\nu}^{-1}u_{\varepsilon}$, it follows that $w_{\varepsilon}$ solves the following wave equation
\begin{equation}\label{eq:1325}
\eta_{\varepsilon} \partial_{t,\nu}^2 w_{\varepsilon}+\sigma_{\varepsilon}\partial_{t,\nu}w_{\varepsilon}-\dive_{x}A_{\varepsilon} \nablao w_{\varepsilon}= f,
\end{equation} 
where the oscillating coefficients are given as the composition $A_{\varepsilon} := \invunf A \unf$, $\eta_{\varepsilon}:= \invunf \eta\unf$ and $\sigma_{\varepsilon}:= \invunf \sigma \unf$.
\end{remark}
Using our abstract homogenization result Theorem \ref{thm:mr} we obtain:
\begin{corollary}\label{cor:1331}
Let $A$, $\eta$, $\sigma$ be given as above and let $(u_{\varepsilon},q_{\varepsilon}) \in L^2_{\nu}(\R; H)$ be the unique solution to (\ref{eq:1345}). Then
\begin{equation*}
\unf (u_{\varepsilon},q_{\varepsilon})\rightharpoonup \brac{\iota_s u, \iota_{s^*} q} \quad \text{weakly in }L^2_{\nu}(\R; H).
\end{equation*}
Above, $\iota_s$ and $\iota_{s^{*}}$ denote the canonical embeddings $\iota_s: \kar(\mathrm{grad}_{\omega})\to H_0$ and $\iota_{s^*}: \kar(\dive_{\omega})\to H_1$, and $(u,q)\in L^2_{\nu}(\R; \kar(\mathrm{grad}_{\omega})\oplus \kar(\dive_{\omega}))$ is the unique solution to
\begin{align*}
& \overline{B}_{\mathsf{hom}}\begin{pmatrix} u \\ q \end{pmatrix}= \begin{pmatrix} \iota_{s}^*f \\ 0 \end{pmatrix}, \\ & B_{\mathsf{hom}}=\begin{pmatrix} \partial_{t,\nu} \iota_s^* \eta \iota_s + \iota_{s}^* \sigma \iota_s \; & 0\\
0 \; & \partial_{t,\nu} \iota_{s^*}^{*} A^{-1} \iota_{s^*}  \end{pmatrix}+ \begin{pmatrix} 0 \; & \overline{\iota_{s}^* \dive_{x} \iota_{s^*}}\\ \overline{\iota^*_{s^*} \nablao \iota_s} \; & 0\end{pmatrix}.
\end{align*}
\end{corollary}
\begin{remark}
In order to recover the classical form of the homogenized wave equation, we set $w= \partial_{t,\nu}^{-1} u$ to obtain (with the usual abuse of notation)
\begin{equation*}
P_{\mathsf{inv}}\eta \partial_{t,\nu}^2 w + P_{\mathsf{inv}}\sigma \partial_{t,\nu}w + \dive_{x} P_{\mathsf{inv}} q = f,
\end{equation*}
where $q$ satisfies $P_{\kar(\dive_{\omega})}\brac{\partial_{t,\nu} A^{-1}q +  \nablao \partial_{t,\nu} w} = 0$ (here $P_{\kar(\dive_{\omega})}=\iota_{s^*}^*$). Applying $\partial_{t,\nu}^{-1}$, it follows that $P_{\kar(\dive_{\omega})} A^{-1} q = - \nablao w$. As a result of this and using (\ref{equation:957}), we have that $A^{-1}q = -\nablao w + \chi$ where $\chi\in L^2(Q)\otimes L^2_{\mathsf{pot}}(\Omega)$ (note that $\chi$ is uniquely determined). Also, since $q \in \kar(\dive_{\omega})$, we have $-\dive_{\omega} A\brac{\nablao w - \chi}=0$. Collecting these facts, we obtain that $w$ satisfies
\begin{equation}\label{eq:1349}
P_{\mathsf{inv}}\eta \partial_{t,\nu}^2 w + P_{\mathsf{inv}}\sigma \partial_{t,\nu}w - \dive_{x} P_{\mathsf{inv}} A \brac{\nablao w - \chi}  = f,
\end{equation} 
where $\chi$ solves the usual corrector equation $-\dive_{\omega} A(\nablao w - \chi)=0$. 
\end{remark}
The above remark suggests the following corrector type statement $\unf \nablao w_{\varepsilon} + \chi \to  \nablao w$ that is equivalent to $A^{-1} \unf q_{\varepsilon} - \chi \to  P_{\kar(\dive_{\omega})}A^{-1}q = A^{-1}q - \chi $. We formalize this in the following.
\begin{proposition}[Corrector type result] Let $f\in H^1_{\nu}(\R; H_0)$, then
\begin{equation*}
\|\partial_{t,\nu}^{-1} u_{\varepsilon}-\partial_{t,\nu}^{-1} u\|_{L^2_{\nu}(\R,H_0)} + \|u_{\varepsilon} - u\|_{L^2_{\nu}(\R; H_0)}+ \|\unf q_{\varepsilon}- q\|_{L^2_{\nu}(\R; H_1)} \to 0 \quad \text{as }\varepsilon \to 0.
\end{equation*}
\begin{proof}
This proof follows similar lines to the proof of Proposition \ref{prop:1277}. Using the assumptions on $A$, $\sigma$, and $\eta$, we have 
\[
\Re \left\langle \begin{pmatrix} w_\epsilon \\ q_\epsilon\end{pmatrix}, \unf \begin{pmatrix} \partial_{t,\nu}\eta+ \sigma \; & 0\\ 0 \; & \partial_{t,\nu} A^{-1}  \end{pmatrix}\invunf\begin{pmatrix} w_\epsilon \\ q_\epsilon\end{pmatrix}\right\rangle \geq c\left\|\begin{pmatrix} w_\epsilon \\ q_\epsilon\end{pmatrix}\right\|^2.
\]
 As a result of this and by the skew-self-adjointness of $\begin{pmatrix} 0 \; & \dive_{x}\\ \nablao \; & 0 \end{pmatrix}$, we obtain
\begin{eqnarray*}
& & \|u_{\varepsilon}-\invunf u\|^{2}_{L^2_{\nu}(\R; H_0)}+\|q_{\varepsilon}-\invunf q\|^2_{L^2_{\nu}(\R; H_1)}\\ & \leq & \frac{1}{c} \Re \expect{B_{\varepsilon} \begin{pmatrix}
u_{\varepsilon}-\invunf u\\ q_{\varepsilon}-\invunf q
\end{pmatrix},\begin{pmatrix}
u_{\varepsilon}-\invunf u\\ q_{\varepsilon}- \invunf q
\end{pmatrix}} \\
& = & \frac{1}{c} \Re\expect{\begin{pmatrix}f\\ 0 \end{pmatrix}, \begin{pmatrix} u_{\varepsilon} - \invunf u\\ q_{\varepsilon}-\invunf q \end{pmatrix}}-\frac{1}{c}\Re \expect{B_{\varepsilon} \begin{pmatrix} \invunf u \\ \invunf q \end{pmatrix}, \begin{pmatrix} u_{\varepsilon}-\invunf u \\ q_{\varepsilon}-\invunf q \end{pmatrix}},
\end{eqnarray*}
where in the second equality we use the equation (\ref{eq:1345}). The first term on the right-hand side vanishes in the limit $\varepsilon\to 0$ using Corollary \ref{cor:1331} and that $\unf f= f$. The second term is treated as follows. We have
\begin{eqnarray*}
&& \Re \expect{B_{\varepsilon}\begin{pmatrix}
\invunf u \\ \invunf q
\end{pmatrix}, \begin{pmatrix} u_{\varepsilon} - \invunf u \\ q_{\varepsilon} - \invunf q \end{pmatrix}}\\
& = & \Re \expect{ \brac{ \begin{pmatrix}
\partial_{t,\nu} \eta + \sigma \; & 0 \\ 0 \; & \partial_{t,\nu} A^{-1}
\end{pmatrix} + \unf \begin{pmatrix} 0 \; & \dive_{x} \\ \nablao \; & 0 \end{pmatrix} \invunf  }\begin{pmatrix} u \\ q  \end{pmatrix}, \begin{pmatrix} \unf u_{\varepsilon} - u \\ \unf q_{\varepsilon} - q   \end{pmatrix} }
\end{eqnarray*}
Similarly as in (\ref{eq:1275}), we obtain $\unf \begin{pmatrix} 0 \; & \dive_{x}\\ \nablao \; & 0 \end{pmatrix}\invunf \begin{pmatrix} u \\ q \end{pmatrix}=\begin{pmatrix} 0 \; & \dive_{x}\\ \nablao \; & 0 \end{pmatrix} \begin{pmatrix} u \\ q \end{pmatrix}$ and therefore the above expression vanishes in the limit $\varepsilon\to 0$. This concludes the proof.
\end{proof}
\end{proposition}
Above, we assumed that the right-hand side $f$ is smooth for convenience, if this is not the case, we could use a mollification procedure as in Proposition \ref{prop:1277}. We briefly summarize the implications of the above proposition for the sequence $w_{\varepsilon}$:
\begin{equation}\label{eq:1382}
w_{\varepsilon}\to w, \quad \partial_{t,\nu} w_{\varepsilon}\to \partial_{t,\nu} w, \quad \nablao w_{\varepsilon}+ \invunf \chi \to \nablao w \quad (\text{strongly in }L^2_{\nu}(\R; L^2(Q)\otimes L^2(\Omega))).
\end{equation}
\begin{remark} It is well-known that the classical homogenized wave equation (\ref{eq:1349}) is an unsatisfactory approximation for (\ref{eq:1325}) on large time scales $t\gtrsim \varepsilon^{-2}$ (see \cite{Dohnal2014}). This fact is not reflected in our result ((\ref{eq:1382}) is given in a norm accounting for the whole time line $t\in \R$) since we use a norm which is weighted with an exponential weight $e^{-2 \nu t}$ ($\nu>0$) that diminishes the effects on large time scales.
\end{remark}

\section*{Acknowledgments}
SN and MV acknowledge funding by the Deutsche Forschungsgemeinschaft (DFG, German
Research Foundation) – project number 405009441, and in the context of TU Dresden's Institutional Strategy ``The Synergetic University''.

\end{document}